\documentclass[11pt]{amsart}

\usepackage{geometry}
\usepackage{xcolor}
\usepackage{amssymb}
\usepackage{amsmath}
\usepackage{arydshln}

\def\BibTeX{{\rm B\kern-.05em{\sc i\kern-.025em b}\kern-.08em
    T\kern-.1667em\lower.7ex\hbox{E}\kern-.125emX}}

\newtheorem{thm}{Theorem}[section]
\newtheorem{cor}[thm]{Corollary}
\newtheorem{lem}[thm]{Lemma}
\newtheorem{prop}[thm]{Proposition}
\newtheorem{exmp}[thm]{Example}
\newtheorem{rem}[thm]{Remark}

\newtheorem*{claim*}{Claim}

\theoremstyle{definition}

\numberwithin{equation}{section}

\usepackage{hyperref}
\usepackage{bbm}
\usepackage{mathrsfs}

\usepackage{pdflscape}

\begin{document}
\title[Exponential ergodicity for SEs of nonnegative processes with jumps]{Exponential ergodicity for stochastic equations of nonnegative processes with jumps}

\author{Martin Friesen}
\address[Martin Friesen]{Fakult\"at f\"ur Mathematik und Naturwissenschaften\\ Bergische Universit\"at Wuppertal\\ 42119 Wuppertal, Germany}
\email[Martin Friesen]{friesen@math.uni-wuppertal.de}

\author[Peng Jin]{Peng Jin\textsuperscript{*}}
\thanks{\textsuperscript{*}Peng Jin is supported by the STU Scientific Research Foundation for Talents (No. NTF18023).}
\address[Peng Jin]{Department of Mathematics \\ Shantou University \\ Shantou, Guangdong 515063, China}
\email{pjin@stu.edu.cn}

\author{Jonas Kremer}
\address[Jonas Kremer]{Fakult\"at f\"ur Mathematik und Naturwissenschaften\\ Bergische Universit\"at Wuppertal\\ 42119 Wuppertal, Germany}
\email{kremer@math.uni-wuppertal.de}

\author{Barbara R\"udiger}
\address[Barbara R\"udiger]{Fakult\"at f\"ur Mathematik und Naturwissenschaften\\ Bergische Universit\"at Wuppertal\\ 42119 Wuppertal, Germany}
\email{ruediger@uni-wuppertal.de}

\date{\today}



\begin{abstract}
In this work, we study ergodicity of continuous time Markov processes on state space $\mathbb{R}_{\geq 0} := [0,\infty)$ obtained as unique strong solutions
to stochastic equations with jumps.
Our first main result establishes exponential ergodicity in the Wasserstein distance,
provided the stochastic equation satisfies a comparison principle and the drift
is dissipative. In particular, it is applicable to continuous-state branching processes with immigration (shorted as CBI processes), possibly with nonlinear branching mechanisms or in L\'evy random environments. Our second main result establishes exponential ergodicity in total variation distance for subcritical CBI processes under a first moment condition on the jump measure for branching and a $\log$-moment condition on the jump measure for immigration.

\end{abstract}

\maketitle

\allowdisplaybreaks

\section{Introduction}

The study of long-time behavior for continuous-time Markov processes is a classical and still popular topic in probability theory. In this paper we will investigate this problem for jump-diffusions on the state space $\mathbb{R}_{\geq 0} := [0,\infty)$, which include interesting classes of processes such as
\emph{continuous-state branching processes with immigration} (see, e.g., \cite{MR2760602,MR3496029}), possibly in L\'evy random environments (see \cite{MR3866603,MR3605717}),
\emph{continuous-state nonlinear branching processes} (see \cite{2017arXiv170801560L}),
and \emph{TCP processes} (see, e.g., \cite{MR3035738,MR2653264}).
All these processes just mentioned belong to the class of Markov processes with state space $\mathbb{R}_{\geq 0}$ whose Markov generator is, for twice continuously differentiable functions $f$ with compact support, i.e., $f \in C_c^2(\mathbb{R}_{\geq 0})$, of the form
\begin{align}
Lf(x) &=
b(x)f'(x) + \frac{1}{2}\int_{E}\sigma(x,u)^2 \varkappa(\mathrm{d}u)f''(x)
+\int_{U_1}\left( f(x + g_1(x,u)) - f(x)\right)\mu_1(\mathrm{d}u) \nonumber \\
 &\quad+ \int_{U_0}\left( f(x + g_0(x,u)) - f(x) - g_0(x,u)f'(x)\right) \mu_0(\mathrm{d}u), \qquad x \geq 0.\label{eq:gen general case}
\end{align}
Here $E,\thinspace U_0,\thinspace U_1$ are complete, separable metric spaces, $\varkappa,\thinspace \mu_0,\thinspace \mu_1$ are $\sigma$-finite measures and $b,\thinspace \sigma,\thinspace g_0,\thinspace g_1$ should satisfy certain restrictions such that the corresponding Markov process exists.
A pathwise construction for this type of Markov processes in terms of strong solutions to stochastic equations were developed in the works of Fu and Li \cite{MR2584896}, Dawson and Li \cite{MR2952093}, and Li and Pu \cite{MR2965746}. Additional related results for
stochastic equations on $\mathbb{R}_{\geq 0}$ can be found in Li and Mytnik \cite{MR2884224} as well as Fournier \cite{MR3060151}.

Let $\lbrace P_t(x,\mathrm{d}y) \thinspace : \thinspace t,\thinspace x \geq 0\rbrace$ be
the transition probabilities of a Markov process with state space $\mathbb{R}_{\geq0}$.
By $\mathcal{P}(\mathbb{R}_{\geq0})$ we denote the space of all Borel probability measures over $\mathbb{R}_{\geq0}$. We call $\pi \in \mathcal{P}(\mathbb{R}_{\geq 0})$ an \emph{invariant distribution} for $\lbrace P_t(x,\mathrm{d}y) \thinspace : \thinspace t,\thinspace x \geq 0\rbrace$, if
\[
\int_{\mathbb{R}_{\geq0}}P_t(x,\mathrm{d}y)\pi\left(\mathrm{d}x\right) = \pi(\mathrm{d}y), \quad t \geq 0.
\]
Existence of invariant distributions is often shown by various compactness arguments, see, e.g., Section 9 in Chapter 4 of \cite{MR838085} for some sufficient conditions.
Unlike existence, uniqueness of the invariant distribution may be a more demanding mathematical problem.
Once existence and uniqueness of an invariant distribution $\pi$ is shown, it is then natural to study the convergence of $P_t(x,\mathrm{d}y)$ to $\pi$.
In order to study such convergence, let us define, for $\varrho,\thinspace \widetilde{\varrho}\in\mathcal{P}(\mathbb{R}_{\geq0})$, the Wasserstein distance
\[
 W_d(\varrho, \widetilde{\varrho}) = \inf \left\{ \int_{\mathbb{R}_{\geq 0} \times \mathbb{R}_{\geq 0}}d(x,y)H(\mathrm{d}x, \mathrm{d}y) \ : \ H \text{ is a coupling of } (\varrho, \widetilde{\varrho}) \right\},
\]
where $d$ is a suitably chosen metric on $\mathbb{R}_{\geq 0}$.
Natural examples for $d$, among others, are $d(x,y) = \mathbbm{1}_{\lbrace x \neq y\rbrace}$ corresponding to the total variation distance and $d(x,y) = |x-y|$ in accordance with the Kantorovich-Rubinstein distance. We will collect some basic properties of $W_d$ in Section \ref{sec:wasserstein distances}, while a detailed treatment of Wasserstein distances is provided in the monograph of Villani \cite{MR2459454}.

We call a Markov process \emph{exponentially ergodic in $W_d$}, if there exists a constant $A > 0$ and a function $K : \mathbb{R}_{\geq 0} \to \mathbb{R}_{\geq 0}$ satisfying
\[
W_d( P_t(x,\cdot), \pi) \leq K(x)\mathrm{e}^{-A t}, \quad t,\thinspace x \geq 0.
\]
A widely used approach for the study of exponential ergodicity in the total variation distance (i.e. $d(x,y) = \mathbbm{1}_{\lbrace x \neq y \rbrace}$) is due to Meyn and Tweedie \cite{MR2509253,MR1234295}.
The essential obstacle when applying their approach lies within the ``irreducibility
of a skeleton chain''.
To prove this, it is sufficient and customary to verify that $P_t(x,\mathrm{d}y)$ has a
jointly continuous density which is strictly positive.
While such an approach is suitable for diffusion processes, the situation is more delicate and
requires a custom-tailored analysis when dealing with Markov processes with jumps. Albeit
being challenging, the approach of Meyn and Tweedie has already been successfully applied to diverse Markov
processes with jumps.
Another approach to prove ergodicity of Markov processes is based on the construction of successful couplings.
Such construction is usually closely related with the mathematical model at hand and often a difficult task, see, e.g., \cite{B14,MR3568041,MR3859440,MR3474827}.

In this work we provide a simple approach to exponential ergodicity in the Wasserstein distance $W_d$ with $d(x,y) = |x-y|$ for Markov processes on $\mathbb{R}_{\geq 0}$ which can be constructed as strong solutions of stochastic equations satisfying the comparison principle. In particular, based on the comparison principle we estimate the trajectories of the Markov process and deduce from that the existence, uniqueness of invariant distributions, and exponential ergodicity in the Wasserstein distance.
The corresponding main result is formulated in Theorem \ref{thm:exponential ergodicity}. Similar ideas have been previously applied in \cite{FJR} to affine processes which include the specific case of continuous-state branching processes with immigration. To illustrate the usage of Theorem \ref{thm:exponential ergodicity}, we apply it to three classes of Markov processes on $\mathbb{R}_{\geq 0}$ which we next explain in more details. We will also discuss limitations and possible improvements of Theorem \ref{thm:exponential ergodicity} in these particular cases.

\subsection{Continuous-state nonlinear branching processes with immigration}

\emph{Continuous-state nonlinear branching processes (CNB process)} were recently introduced and studied by Li \textit{et al.} (2017) \cite{2017arXiv170801560L}. In this paper we add to the CNB process a general immigration mechanism and therefore call it a \emph{continuous-state nonlinear branching process with immigration (CNBI process)}.
The corresponding Markov generator for the class of CNBI processes is, for $f\in C_c^2(\mathbb{R}_{\geq0})$, of the form
\begin{align}\label{EQ:NONLINEARCBI}
 Lf(x) &= \gamma_0(x) f'(x) + \frac{\gamma_1(x)}{2}f''(x) \nonumber \\
&\quad + \gamma_2(x) \int_{\mathbb{R}_{\geq0}}(f(x+z) - f(x) - z f'(x))m(\mathrm{d}z) + \int_{\mathbb{R}_{\geq0}}\left(f(x+z)-f(x)\right)\nu(\mathrm{d}z),
\end{align}
where $\gamma_0,\thinspace \gamma_1,\thinspace \gamma_2$ are Borel-functions on $\mathbb{R}_{\geq0}$, $\gamma_0(0) \geq 0$ satisfying $\gamma_1, \gamma_2 \geq 0$, and $m$, $\nu$ are Borel measures on $(0,\infty)$ satisfying
\begin{equation}
\int_{\mathbb{R}_{\geq 0}} (z \wedge z^2) m(\mathrm{d}z) + \int_{\mathbb{R}_{\geq 0}}(1 \wedge z) \nu(\mathrm{d}z) < \infty.\label{eq:measure conditions}
\end{equation}
If $\gamma_0,\thinspace \gamma_1,\thinspace \gamma_2$ are locally Lipschitz continuous on $(0,\infty)$ and $\nu\equiv0$, then a pathwise construction of the corresponding Markov process (called a CNB process) is established in \cite{2017arXiv170801560L}.
More precisely, the authors identified the CNB process with a unique strong solution to a certain stochastic equation with the additional restriction that $0$ and $\infty$ are traps.
The last requirement is obligatory to treat particular cases as, e.g., $\gamma_2(x) = x^p$ with $p>0$.
The authors then studied extinction, explosion, and coming down from infinity behaviors of the CNB process.
In Section \ref{sec:CNB} we provide some simple sufficient conditions on the parameters $\gamma_0,\thinspace \gamma_1,\thinspace \gamma_2$ such that the CNBI process is exponentially ergodic in the Wasserstein distance $W_1$ (see Theorem \ref{thm:exponential ergodicity in wassterstein distance for nonlinear cbi}).
To the best of our knowledge, it is the first ergodicity result for general CNBI processes.

\subsection{Continuous-state branching processes with immigration}

\emph{Continuous-state branching processes with immigration (CBI processes)}  are particular cases of CNBI processes.
They have been first introduced by Feller (1950) \cite{MR0046022} and Ji\v{r}ina (1958) \cite{MR0101554} and then developed by Kawazu and Watanabe (1971) \cite{doi:10.1137/1116003}.
For a detailed treatment of CBI processes encompassing a concise introduction we refer to
the monographs of Li \cite{MR2760602} and Pardoux \cite{MR3496029}.
Following \cite{doi:10.1137/1116003}, CBI processes are Feller processes whose Markov generator is, for $f \in C_c^2(\mathbb{R}_{\geq 0})$, given by
\begin{align}
Lf(x)
&=
\left(\beta-bx\right)f'(x)+ \frac{\sigma^2}{2} xf''(x) \nonumber \\
&\quad +
x\int_0^{\infty}\left(f(x+z)-f(x)-zf'(x)\right)m\left(\mathrm{d}z\right) +
\int_0^{\infty}\left(f(x+z)-f(x)\right)\nu\left(\mathrm{d}z\right).\label{eq:generator cbi}
\end{align}
Here $(\beta, b, \sigma, m,\nu)$ are \emph{admissible parameters} in the sense that $\beta \geq 0$, $b \in \mathbb{R}$, $\sigma \geq 0$, $m$ and $\nu$ satisfy \eqref{eq:measure conditions}.
In Section \ref{sec:CBI} we also briefly recall other characterizations of CBI processes
in terms of their Laplace transforms and strong solutions to stochastic equations with jumps motivated by \cite{MR2243880,MR2584896}.

Previously, a number of authors investigated the long-time behavior of CBI processes.
Pinsky \cite{MR0295450} announced the existence of a limit distribution for subcritical ($b>0$) CBI processes under the condition
\begin{align}\label{eq:inequality w_log pi}
\int_{\{ z > 1\}}\log(z)\nu\left(\mathrm{d}z\right)<\infty.
\end{align}
It was shown subsequently in \cite[Theorem 3.20 and Corollary 3.21]{MR2760602} and \cite[Theorem 3.16]{MR2779872} that for subcritical CBI processes condition \eqref{eq:inequality w_log pi} is equivalent to the weak convergence of the associated transition probabilities towards a limiting distribution. The limit distribution was also shown to be the unique invariant distribution for the CBI process.
Properties of this distribution were investigated in \cite{MR2922631}.
A multidimensional version of Pinsky's result was recently studied in \cite{2018arXiv181205402J}, while in \cite{FJR} exponential ergodicity in different Wasserstein distances
was derived under reasonable integrability conditions on $\nu$.

In the setting of CBI processes, our Theorem \ref{thm:exponential ergodicity} is applicable but the obtained result is not particularly strong, compared with \cite[Theorem 1.6]{FJR}. On the other hand, the exponential ergodicity in the total variation distance
is yet under current investigation.
Based on the approach of Meyn and Tweedie, particular examples
have been studied in
\cite{2017arXiv170900969J,MR3451177,MR3437080,MR2287102} using the condition that
\begin{align}\label{eq:inequality p pi}
\exists \  \varepsilon > 0 \text{ such that } \int_{\{ z > 1\}}z^{\varepsilon} \nu\left(\mathrm{d}z\right)<\infty
\end{align}
in order to derive a Lyapunov drift criteria inherent in the approach of Meyn and Tweedie.
An alternative approach based on the construction of a successful coupling  was recently established by Li and Ma \cite{MR3343292}.
Following \cite[Theorem 2.5]{MR3343292}, a subcritical CBI process with admissible parameters $(b,\beta,\sigma,m,\nu\equiv0)$ is exponentially ergodic in the total variation distance, provided Grey's condition on the immigration mechanism is satisfied (see condition (5.a) below).

In Section \ref{sec:CBI} we extend the aforementioned results and establish exponential ergodicity in total variation distance for general subcritical CBI processes (including the case $\nu\not\equiv0$) under the weaker integrability condition \eqref{eq:inequality w_log pi} (compared with \eqref{eq:inequality p pi}).
The corresponding main result is formulated in Theorem \ref{thm:exponential ergodicity cbi} and it seems to the first result establishing a convergence rate in total variation distance merely under \eqref{eq:inequality w_log pi}.

Exponential ergodicity often plays an essential role in statistical estimation of the parameters of the underlying process, as demonstrated in several articles, see \cite{MR3769658,MR3175784,MR2995525,MR3343292} and the references therein. Eventually, to illustrate some applications of Theorem \ref{thm:exponential ergodicity cbi}, we add at the end of section 5 a strong law of large numbers as well as a functional central limit theorem for CBI processes (see Corollaries \ref{cor:slln} and \ref{cor:FCLT} below).

\subsection{Continuous-state branching processes with immigration in the L\'evy random environment}

A \emph{continuous-state branching process with immigration in a L\'evy random environment (CBIRE process)} is a Markov process on $\mathbb{R}_{\geq 0}$ with generator $L = L_0 + L_1$ acting on $C_c^2(\mathbb{R}_{\geq0})$, where $L_0$ is given by \eqref{eq:generator cbi} and
\begin{align}
 L_1f(x) &= b_E x f'(x) + \frac{\sigma_E^2}{2} x^2 f''(x) + \int_{[-1,1]^c}\left( f(xe^z) - f(x)\right) \mu_E(\mathrm{d}z) \nonumber \\
 &\quad
+ \int_{[-1,1]}\left( f(xe^z) - f(x) - x(e^z - 1)f'(x) \right)\mu_E(\mathrm{d}z).\label{eq:generator  CBIRE}
\end{align}
Here, $b_E \in \mathbb{R}$, $\sigma_{E} \geq 0$ and $\mu_E$ is a L\'evy measure on $\mathbb{R}$.
We refer to He \textit{et al.} \cite{MR3866603} for the general theory and a comprehensive introduction of CBIRE processes, see also the works of Palau and Pardo \cite{MR3605717,MR3745730}. In \cite{MR3866603} the authors gave a necessary and sufficient condition for the existence of invariant distributions of CBIRE processes. It is our aim in Section \ref{sec:CBIRE} to prove ergodicity in both the Wasserstein and total variation distance for these processes (see Theorems \ref{thm:ergodicity in Wasserstein distance CBIRE} and \ref{thm:ergodicity in tv distance CBIRE} below). Those results are obtained with the help of our previously obtained ergodic results in Sections \ref{sec:Stochastic equations of nonnegative processes with jumps} and \ref{sec:CBI}.

\subsection{Structure of the work}
In Section \ref{sec:wasserstein distances} we recall some properties of Wasserstein distances.
Exponential ergodicity in the Wasserstein distance for Markov processes
with generator \eqref{eq:gen general case} is established
in Section \ref{sec:Stochastic equations of nonnegative processes with jumps}.
The particular example of CNBI processes is then discussed in Section \ref{sec:CNB}.
Next, in Section \ref{sec:CBI} we study the exponential ergodicity in total variation distance for CBI processes.
Finally, in Section \ref{sec:CBIRE} an application of our previous results to CBIRE processes is given.

\section{Some basic properties of Wasserstein distances}
\label{sec:wasserstein distances}

By $\mathcal{P}(\mathbb{R}_{\geq0})$ we denote the space of all Borel probability measures over $\mathbb{R}_{\geq0}$.
Given $\varrho,\thinspace \widetilde{\varrho}\in\mathcal{P}(\mathbb{R}_{\geq0})$, a coupling $H$ of $(\varrho,\widetilde{\varrho})$ is a Borel probability measure on $\mathbb{R}_{\geq0}\times\mathbb{R}_{\geq0}$ which has marginals $\varrho$ and $\widetilde{\varrho}$, respectively.
We write $\mathcal{H}(\varrho,\widetilde{\varrho})$ for the collection of all such couplings. Let $d$ be a metric on $\mathbb{R}_{\geq 0}$ such that $(\mathbb{R}_{\geq 0}, d)$ is a complete separable metric space and define
\[
\mathcal{P}_{d}\left(\mathbb{R}_{\geq0}\right)
=\left\lbrace \varrho\in\mathcal{P}\left(\mathbb{R}_{\geq0}\right)\thinspace:\thinspace \int_{\mathbb{R}_{\geq 0}} d(x,0)\varrho\left(\mathrm{d}x\right)<\infty\right\rbrace.
\]
The \emph{Wasserstein distance} on $\mathcal{P}_{d}(\mathbb{R}_{\geq0})$ is defined by
\begin{align}\label{WASSERSTEIN}
W_{d}\left(\varrho,\widetilde{\varrho}\right)=\inf\left\lbrace\int_{\mathbb{R}_{\geq0}\times\mathbb{R}_{\geq0}}d(x,y) H\left(\mathrm{d}x,\mathrm{d}y\right)\thinspace:\thinspace H\in\mathcal{H}\left(\varrho,\widetilde{\varrho}\right)\right\rbrace.
\end{align}
Note that, since $\varrho$ and $\widetilde{\varrho}$ belong to $\mathcal{P}_d(\mathbb{R}_{\geq0})$, the expression $W_d(\varrho,\widetilde{\varrho})$ is finite.
Moreover, it can be shown that this infimum is attained
(see \cite[p.95]{MR2459454}), i.e., there exists $H\in\mathcal{H}(\varrho,\widetilde{\varrho})$ such that
\begin{align}\label{EQ:INFATTAINED}
W_d\left(\varrho,\widetilde{\varrho}\right)=\int_{\mathbb{R}_{\geq0}\times\mathbb{R}_{\geq0}}d(x,y)H\left(\mathrm{d}x,\mathrm{d}y\right).
\end{align}
Since $(\mathbb{R}_{\geq 0}, d)$ is supposed to be a complete separable metric space, according to \cite[Theorem 6.16]{MR2459454}, $(\mathcal{P}_{d}(\mathbb{R}_{\geq0}),W_{d})$ is also a complete separable metric space.
In the remainder of the article, we will use the following particular examples.

\begin{exmp} $ $
\begin{enumerate}
\item[(a)] If $d_{TV}(x,y) = \mathbbm{1}_{ \{ x \neq y\} }$, then
$\mathcal{P}_{d_{TV}}(\mathbb{R}_{\geq 0}) = \mathcal{P}(\mathbb{R}_{\geq 0})$ and
\[
W_{d_{TV}}(\varrho, \widetilde{\varrho}) = \frac{1}{2}\| \varrho - \widetilde{\varrho}\|_{TV} := \frac{1}{2}\sup \left\{ | \varrho(A) - \widetilde{\varrho}(A) | \ : \ A \subset \mathbb{R} \text{ Borel set} \right\}
\]
is the total variation distance.
\item[(b)] The Wasserstein-1-distance corresponds to $d_1(x,y) = |x-y|$, where
\[
\mathcal{P}_{d_1}(\mathbb{R}_{\geq 0}) := \mathcal{P}_1(\mathbb{R}_{\geq 0}) := \left\{ \varrho \in \mathcal{P}(\mathbb{R}_{\geq 0}) \ : \ \int_{\mathbb{R}_{\geq 0}}x \varrho(\mathrm{d}x) < \infty \right\}.
\]
In this case we adopt the shorthand $W_1 := W_{d_1}$.
\item[(c)] If $d_{\log}(x,y) = \log(1 + |x-y|)$, then
\[
\mathcal{P}_{d_{\log}}(\mathbb{R}_{\geq 0}) := \mathcal{P}_{\log}\left(\mathbb{R}_{\geq0}\right)
:=\left\lbrace \varrho\in\mathcal{P}\left(\mathbb{R}_{\geq0}\right)\thinspace:\thinspace \int_{\lbrace x>1\rbrace}\log (x)\varrho\left(\mathrm{d}x\right)<\infty\right\rbrace
\]
and $W_{\log} := W_{d_{\log}}$ is suited for CBI processes studied in Section \ref{sec:CBI}.
\end{enumerate}
\end{exmp}

One simple but important property of Wasserstein distances is their convexity
as formulated below.

\begin{lem}\label{lem:convexity of wasserstein distances}
Let $d$ be a metric such that $(\mathbb{R}_{\geq 0},d)$ is a complete separable metric space. Let $\varrho,\thinspace \widetilde{\varrho} \in \mathcal{P}_d(\mathbb{R}_{\geq 0})$ and suppose that $P_t(x,\mathrm{d}y)$ is a Markov kernel on $\mathbb{R}_{\geq 0}$.
Then, for any $H \in \mathcal{H}(\varrho, \widetilde{\varrho})$, we have
\[
W_d\left( \int_{\mathbb{R}_{\geq 0}}P(x,\cdot)\varrho(\mathrm{d}x), \int_{\mathbb{R}_{\geq 0}}P(x,\cdot)\widetilde{\varrho}(\mathrm{d}x)\right)
\leq \int_{\mathbb{R}_{\geq 0}\times \mathbb{R}_{\geq 0}} W_d(P(x,\cdot), P(y,\cdot)) H(\mathrm{d}x, \mathrm{d}y).
\]
\end{lem}

For a proof we refer the reader to \cite[Theorem 4.8]{MR2459454}. The convolution between measures $\varrho$ and $\widetilde{\varrho}$ on $\mathbb{R}_{\geq 0}$ is denoted by $\varrho \ast \widetilde{\varrho}$.
We close the presentation with a useful convolution estimate for Wasserstein distances.
\begin{lem}\label{lem:convolution estimate} $ $
Let $d$ be a metric such that $(\mathbb{R}_{\geq 0},d)$ is a complete separable metric space. Suppose that
\[
d(x+y, \widetilde{x} + y) \leq d(x,\widetilde{x}), \qquad x,\widetilde{x},y \geq 0.
\]
Let $\varrho,\thinspace \widetilde{\varrho}, \thinspace g \in\mathcal{P}_{d}(\mathbb{R}_{\geq0})$. Then $
W_{d}(\varrho \ast g,\widetilde{\varrho}\ast g)\leq W_{d}( \varrho,\widetilde{\varrho})$.
\end{lem}

\begin{proof}
We define $\Vert h \Vert_{\mathrm{Lip}} = \sup_{x \neq y} \frac{|h(x) - h(y)|}{d(x,y)}$.
Using the Kantorovich-Duality we obtain
\begin{align*}
W_{d}(\varrho \ast g, \widetilde{\varrho}\ast g) &= \sup_{\| h \|_{\mathrm{Lip}} \leq 1} \left| \int_{\mathbb{R}_{\geq 0}} h(x) (\varrho\ast g)(\mathrm{d}x) - \int_{\mathbb{R}_{\geq 0}}h(x)(\widetilde{\varrho}\ast g)(\mathrm{d}x)\right| \\
&= \sup_{\| h \|_{\mathrm{Lip}} \leq 1} \left| \int_{\mathbb{R}_{\geq 0}} h_g(x) \varrho(\mathrm{d}x) - \int_{\mathbb{R}_{\geq 0}}h_g(x)\widetilde{\varrho}(\mathrm{d}x)\right| \\
&\leq \sup_{\| h \|_{\mathrm{Lip}} \leq 1} \left| \int_{\mathbb{R}_{\geq 0}} h(x) \varrho(\mathrm{d}x) - \int_{\mathbb{R}_{\geq 0}}h(x) \widetilde{\varrho}(\mathrm{d}x)\right| = W_d(\varrho, \widetilde{\varrho}),
\end{align*}
where we used that $h_g(x)=\int_{\mathbb{R}_{\geq0}}h(x+y)g\left(\mathrm{d}y\right)$ satisfies $\| h_g\|_{\mathrm{Lip}} \leq 1$.
\end{proof}

Although we formulated Lemma \ref{lem:convexity of wasserstein distances} and \ref{lem:convolution estimate} on the state space $\mathbb{R}_{\geq 0}$,
it is clear that they naturally extend to more abstract state spaces.
In particular, Lemma \ref{lem:convolution estimate} can be easily obtained for
arbitrary convex cones.

\section{Stochastic equations of nonnegative processes with jumps}
\label{sec:Stochastic equations of nonnegative processes with jumps}


Let $E$, $U_0$, and $U_1$ be complete separable metric spaces. Following Dawson and Li \cite{MR2952093}, we say that the parameters $(b,\sigma,g_0,g_1)$ are \emph{admissible} if:
\begin{itemize}
\item $b(0)\geq0$ and $b(x)=b_1(x)-b_2(x)$ is defined on $\mathbb{R}_{\geq0}$, where $x\mapsto b_1(x)$ is a continuous function, and $x\mapsto b_2(x)$ is a continuous and nondecreasing function;
\item $(x,u)\mapsto\sigma(x,u)$ is a Borel function on $\mathbb{R}_{\geq0}\times E$ satisfying $\sigma(0,u)=0$ for $u\in E$;
\item $(x,u)\mapsto g_0(x,u)$ is a Borel function on $\mathbb{R}_{\geq0}\times U_0$ satisfying $g_0(0,u)=0$ and $g_0(x,u)+x\geq0$ for $x>0$ and $u\in U_0$;
\item $(x,u)\mapsto g_1(x,u)$ is a Borel function on $\mathbb{R}_{\geq0}\times U_1$ satisfying $g_1(x,u)+x\geq0$ for $x\in\mathbb{R}_{\geq0}$ and $u\in U_1$.
\end{itemize}
Let $(\Omega,\mathcal{F},(\mathcal{F}_t)_{t\geq0},\mathbb{P})$ be a filtered probability space satisfying the usual hypotheses, i.e. $(\Omega,\mathcal{F},\mathbb{P})$ is complete, the filtration $(\mathcal{F}_t)_{t\geq0}$ is right-continuous, and $\mathcal{F}_0$ contains all $\mathbb{P}$-null sets in $\mathcal{F}$.
Let $\varkappa(\mathrm{d}z)$, $\mu_0(\mathrm{d}u)$, and $\mu_1(\mathrm{d}u)$ be $\sigma$-finite measures on $E$, $U_0$, and $U_1$, respectively. Let $W(\mathrm{d}t,\mathrm{d}u)$ be a $(\mathcal{F}_t)_{t\geq0}$-Gaussian white noise on $\mathbb{R}_{\geq0}\times E$  with intensity measure $\mathrm{d}t\varkappa(\mathrm{d}z)$, and let $ N_0(\mathrm{d}t,\mathrm{d}u)$ and $N_1(\mathrm{d}t,\mathrm{d}u)$ be $(\mathcal{F}_t)_{t\geq0}$-Poisson random measures on $\mathbb{R}_{\geq0}\times U_0$ and $\mathbb{R}_{\geq0}\times U_1$ with intensities $\mu_0(\mathrm{d}u)$ and $\mu_1(\mathrm{d}u)$, respectively. Denote by $\widetilde{N}_0(\mathrm{d}t,\mathrm{d}u):=N_0(\mathrm{d}t,\mathrm{d}u)-\mathrm{d}t\mu_0(\mathrm{d}u)$ the compensated Poisson random measure of $N_0(\mathrm{d}t,\mathrm{d}u)$.
Suppose that the random objects $W(\mathrm{d}s,\mathrm{d}u)$, $N_0(\mathrm{d}s,\mathrm{d}u)$, and $N_1(\mathrm{d}s,\mathrm{d}u)$ are mutually independent.

We consider a stochastic process $\lbrace X_t\thinspace:\thinspace t\geq0\rbrace$ with state space $\mathbb{R}_{\geq0}$ determined by the stochastic equation
\begin{align}
X_t &=
X_0+\int_0^t b\left(X_s\right)\mathrm{d}s+\int_0^t\int_E\sigma\left(X_s,u\right)W\left(\mathrm{d}s,\mathrm{d}u\right) \nonumber \\
&\quad+
\int_0^t\int_{U_0}g_0\left(X_{s-},u\right)\widetilde{N}_0\left(\mathrm{d}s,\mathrm{d}u\right)+
\int_0^t\int_{U_1}g_1\left(X_{s-},u\right)N_1\left(\mathrm{d}s,\mathrm{d}u\right),\quad t\geq0,\label{eq:se}
\end{align}
where $X_0\geq0$ is $\mathcal{F}_0$-measurable.
A \emph{strong solution} to \eqref{eq:se} is, by definition,
a nonnegative c\`adl\`ag and $(\mathcal{F}_t)_{t\geq0}$-adapted\footnote{Adapted to the augmented natural filtration generated by $W(\mathrm{d}s,\mathrm{d}u)$, $N_0(\mathrm{d}s,\mathrm{d}u)$, and $N_1(\mathrm{d}s,\mathrm{d}u)$, see, e.g., Situ \cite[p.76]{MR2160585}.} process $\lbrace X_t\thinspace:\thinspace t\geq0\rbrace$ satisfying the equation \eqref{eq:se} almost surely for every $t\geq0$.
Existence and uniqueness of strong solutions to \eqref{eq:se} were studied by Dawson and Li \cite{MR2952093}, respectively, where the following conditions have been introduced:
\begin{enumerate}
\item[(3.a)] there is a constant $K\geq0$ so that
\[
|b(x)|+\int_{U_1}\left\vert g_1(x,u)\right\vert\mu_1\left(\mathrm{d}u\right)\leq K(1+x), \quad x\geq0;
\]
\item[(3.b)] for each $u\in U_1$ the function $x\mapsto g_1(x,u)+x$ is nondecreasing and for each $m\geq1$ there is a nondecreasing concave function $z\mapsto r_m(z)$ on $\mathbb{R}_{\geq0}$ such that $\int_{0+}r_m(z)^{-1}\mathrm{d}z=\infty$ and
\[
\left\vert b_1(x)-b_1(y)\right\vert+\int_{U_1}\left\vert g_1(x,u)-g_1(y,u)\right\vert\mu_1\left(\mathrm{d}u\right)
\leq r_m\left(\vert x-y\vert\right),\quad 0\leq x,\thinspace y\leq m;
\]
\item[(3.c)] for each $u\in U_0$ the function $g_0(x,u)$ is nondecreasing, and for each $m\geq1$ there is a nonnegative and nondecreasing function $z\mapsto\rho_m(z)$ on $\mathbb{R}_{\geq0}$ so that $\int_{0+}\rho_m(z)^{-2}\mathrm{d}z=\infty$ and
\begin{align*}
\int_E&\left\vert\sigma(x,u)-\sigma(y,u)\right\vert^2\varkappa\left(\mathrm{d}u\right)
+
\int_{U_0}\left(\left\vert g_0(x,u)-g_0(y,u)\right\vert\wedge \left\vert g_0(x,u)-g_0(y,u)\right\vert^2\right)\mu_0\left(\mathrm{d}u\right) \\
&\quad \leq \rho_m\left(\vert x-y\vert\right)^2,\qquad 0\leq x,\thinspace y\leq m.
\end{align*}
\end{enumerate}
The next result summarizes Theorems 2.3 and 2.5 of Dawson and Li \cite{MR2952093}.

\begin{prop}\label{prop:existence and uniqueness of strong solutions}
Suppose that $(b,\sigma,g_0,g_1)$ are admissible parameters satisfying conditions (3.a)-(3.c).
\begin{enumerate}
\item[(a)] Then, for each $\mathcal{F}_0$-measurable random variable with $X_0\geq0$ almost surely, there exists a unique strong solution $\lbrace X_t\thinspace:\thinspace t\geq0\rbrace$ to \eqref{eq:se}.
\item[(b)] Let $X_0$ and $Y_0$ be two $\mathcal{F}_0$-measurable nonnegative random variables. Denote by $\lbrace X_t\thinspace:\thinspace t\geq0\rbrace$ and $\lbrace Y_t\thinspace:\thinspace t\geq0\rbrace$ the corresponding strong solutions to \eqref{eq:se}. If $\mathbb{P}(X_0\leq Y_0)=1$, then $\mathbb{P}(X_t\leq Y_t \text{ for all }t\geq0)=1$.
\end{enumerate}
\end{prop}
In the following, we write $\lbrace X_t^x\thinspace:\thinspace t\geq0\rbrace$ for the unique strong solution of \eqref{eq:se} to indicate that the process $X_t$ starts with initial variable $X_0=x\geq0$ almost surely.
Denote by $\mathcal{B}_b(\mathbb{R}_{\geq0})$ the Banach space of all real-valued, bounded and Borel-measurable functions on $\mathbb{R}_{\geq0}$ endowed with the norm $\Vert f\Vert_{\infty}:=\sup_{x\in\mathbb{R}_{\geq0}}\vert f(x)\vert$.
Arguinig as in \cite{LI2018} (see the proof of Theorem 1.1 therein) one can show that
$\lbrace X_t^x\thinspace:\thinspace t\geq0\rbrace$ is a strong $(\mathcal{F}_t)_{t\geq0}$-Markov process and has transition probabilities $P_t(x,\mathrm{d}y)$, i.e., it holds
\[
P_tf(x):=\mathbb{E}\left[f\left(X_t^x\right)\right]=\int_{\mathbb{R}_{\geq0}} f(y)P_t\left(x,\mathrm{d}y\right),\quad f\in\mathcal{B}_b(\mathbb{R}_{\geq0}).
\]
Applying It\^o's formula one finds that $\lbrace X_t^x\thinspace:\thinspace t\geq0\rbrace$ solves the local martingale problem with generator $L$ given by \eqref{eq:gen general case} and domain $C_c^2(\mathbb{R}_{\geq0})$.
The adjoint transition semigroup on $\mathcal{P}(\mathbb{R}_{\geq0})$ is defined by
\[
P_t^{\ast}\varrho\left(\mathrm{d}y\right)=\int_{\mathbb{R}_{\geq0}}P_t\left(x,\mathrm{d}y\right)\varrho\left(\mathrm{d}x\right),\quad t\geq0,\thinspace\varrho\in\mathcal{P}\left(\mathbb{R}_{\geq0}\right).
\]
By the Markov property we have that $P_{t+s}^{\ast}=P_t^{\ast}P_s^{\ast}$ for all $0\leq s\leq t$.
Let us formulate the following conditions in addition to (2.a)-(2.c):
\begin{enumerate}
\item[(3.d)] it holds that $\lbrace M_t\thinspace:\thinspace t\geq0\rbrace$ defined by
\[
M_t=\int_0^t\int_E \sigma\left(X_s^x,u\right)W\left(\mathrm{d}s,\mathrm{d}u\right)+\int_0^t\int_{U_0} g_0\left(X_{s-}^x,u\right)\widetilde{N}_0\left(\mathrm{d}s,\mathrm{d}u\right),\quad t\geq0,
\]
is a martingale with respect to the filtration $(\mathcal{F}_t)_{t\geq0}$;
\item[(3.e)] there exists a constant $A>0$ such that, for $\widetilde{b}(x):=b(x)-\int_{U_1}g_1(x,u)\mu_1(\mathrm{d}u)$, we have
\[
\widetilde{b}(y)-\widetilde{b}(x)\leq-A(y-x),\quad 0\leq x\leq y.
\]
\end{enumerate}
Under the given conditions (3.a)-(3.e) we are able to show
that the corresponding Markov process is exponentially ergodic
in the Wasserstein distance $W_1$.

\begin{thm}\label{thm:exponential ergodicity}
Let $(b,\sigma,g_0,g_1)$ be admissible parameters satisfying conditions (3.a)-(3.e). Then, for all $\varrho,\thinspace\widetilde{\varrho}\in\mathcal{P}_1(\mathbb{R}_{\geq0})$, we have
\begin{equation}
W_1\left(P_t^{\ast}\varrho,P_t^{\ast}\widetilde{\varrho}\right)\leq \mathrm{e}^{-At} W_1\left(\varrho,\widetilde{\varrho}\right),\quad t\geq0.\label{eq:exp ergodicity}
\end{equation}
In particular, there exists a unique invariant distribution
$\pi\in\mathcal{P}(\mathbb{R}_{\geq0})$.
Moreover, we have $\pi\in\mathcal{P}_1(\mathbb{R}_{\geq0})$ and, for all $\varrho\in\mathcal{P}_1(\mathbb{R}_{\geq0})$,
\begin{align*}
W_1\left(P_t^{\ast}\varrho,\pi\right)\leq \mathrm{e}^{-At}W_1\left(\varrho,\pi\right),\quad t\geq0.
\end{align*}
\end{thm}

\begin{proof}
The proof is divided into several steps.

\textit{Step 1:} Let $\lbrace X_t^x\thinspace:\thinspace t\geq0\rbrace$ and $\lbrace X_t^y\thinspace:\thinspace t\geq0\rbrace$ be strong solutions of \eqref{eq:se} with $0\leq x\leq y$.
Using Proposition \ref{prop:existence and uniqueness of strong solutions} (b) together with (2.d) and (2.e), we obtain
\begin{align*}
\mathbb{E}\left[\left\vert X_t^x-X_t^y\right\vert\right]
&=
\mathbb{E}\left[X_t^y\right]-\mathbb{E}\left[X_t^x\right] \\
&=
y-x+\int_0^t\mathbb{E}\left[\widetilde{b}\left(X_s^y\right)-\widetilde{b}\left(X_s^x\right)\right]\mathrm{d}s \\
&\leq
\vert x-y\vert-A\int_0^t\mathbb{E}\left[\left\vert X_s^x-X_s^y\right\vert\right]\mathrm{d}s,\quad t\geq0.
\end{align*}
Applying Gronwall's lemma yields
\begin{equation}
\mathbb{E}\left[\left\vert X_t^x-X_t^y\right\vert\right]\leq\vert x-y\vert\mathrm{e}^{-At},\quad t\geq0.\label{eq:moment estimate}
\end{equation}

\textit{Step 2:} Let us prove \eqref{eq:exp ergodicity}. We denote by $\delta_x$ and $\delta_y$ the Dirac measure concentrated in $x$ and $y$, respectively. Assume first $0\leq x\leq y$. Since the joint distribution of $(X_t^x,X_t^y)$ belongs to $\mathcal{H}(P_t^{\ast}\delta_x,P_t^{\ast}\delta_y)$, we obtain from \eqref{eq:moment estimate},
\begin{equation}
W_1\left(P_t^{\ast}\delta_x,P_t^{\ast}\delta_y\right)
\leq
\mathbb{E}\left[\left\vert X_t^x-X_t^y\right\vert\right]
\leq
\vert x-y\vert\mathrm{e}^{-At},\quad t\geq0. \label{eq:exp ergodicity with dirac measure}
\end{equation}
Let now $H$ be any coupling of $(\varrho,\widetilde{\varrho})$
satisfying \eqref{EQ:INFATTAINED}.
Using the convexity of $W_1$ and \eqref{eq:exp ergodicity with dirac measure}, we get
\begin{align*}
W_1\left(P_t^{\ast}\varrho,P_t^{\ast}\widetilde{\varrho}\right)
&\leq
\int_{\mathbb{R}_{\geq0}\times\mathbb{R}_{\geq0}} W_1\left(P_t^{\ast}\delta_x,P_t^{\ast}\delta_y\right)H\left(\mathrm{d}x,\mathrm{d}y\right) \\
&\leq
\mathrm{e}^{-At}\int_{\mathbb{R}_{\geq0}\times\mathbb{R}_{\geq0}} \left\vert x-y\right\vert H\left(\mathrm{d}x,\mathrm{d}y\right)
= e^{-At}W_1(\varrho, \widetilde{\varrho}).
\end{align*}

\textit{Step 3:} We prove existence of $\pi$. We fix any $\varrho\in\mathcal{P}_1(\mathbb{R}_{\geq0})$. Then, for $k,l\in\mathbb{N}$ with $k<l$,
\[
W_1\left(P_k^{\ast}\varrho,P_l^{\ast}\varrho\right)
\leq
\sum_{s=k}^{l-1}W_1\left(P_{s+1}^{\ast}\varrho,P_s^{\ast}\varrho\right)
\leq
\sum_{s=k}^{l-1}\mathrm{e}^{-As}W_1\left(P_1^{\ast}\varrho,\varrho\right),
\]
where we have used the semigroup property of $P_{s+1}^{\ast}$ together with \eqref{eq:exp ergodicity}. Since the right-hand side tends to zero as $k,l\to\infty$,  $(P_k^{\ast}\varrho)_{k\in\mathbb{N}}\subset\mathcal{P}_1(\mathbb{R}_{\geq0})$ is a Cauchy sequence. As a consequence, there exists $\pi\in\mathcal{P}_1(\mathbb{R}_{\geq0})$ such that $W_1(P_k^{\ast}\varrho,\pi)$ converges to zero as $k\to\infty$.

We proceed to show invariance of $\pi$, i.e. $P_h^{\ast}\pi=\pi$ for all $h>0$. Fix $h>0$. Using the semigroup property and \eqref{eq:exp ergodicity} it follows
\begin{align*}
W_1\left(P_h^{\ast}\pi,\pi\right)
&\leq
W_1\left(P_h^{\ast}\pi,P_h^{\ast}P_k^{\ast}\varrho\right) + W_1\left(P_k^{\ast}P_h^{\ast}\varrho,P_k^{\ast}\varrho\right) + W_1\left(P_k^{\ast}\varrho,\pi\right) \\
&\leq
\mathrm{e}^{-A h}W_1\left(\pi,P_k^{\ast}\varrho\right) + \mathrm{e}^{-Ak}W_1\left( P_h^{\ast}\varrho,\varrho\right) + W_1\left( P_k^{\ast}\varrho,\pi\right),
\end{align*}
and the right-hand side tends to zero as $k\to\infty$. Hence, we see that $W_1(P_h^{\ast}\pi,\pi)=0$.

\textit{Step 4, uniqueness of $\pi$:} Let $\widehat{\pi}\in\mathcal{P}(\mathbb{R}_{\geq0})$ be any invariant distribution. Let $W_1^{\leq 1}$ be the Wasserstein distance given by \eqref{WASSERSTEIN} with $d(x,y) = 1 \wedge |x-y|$.
Using the invariance of $\pi$, $\widehat{\pi}$, and the convexity of $W_1^{\leq 1}$, for any $H\in\mathcal{H}(\pi,\widehat{\pi})$, we derive
\begin{align*}
W_1^{\leq1}\left(\pi,\widehat{\pi}\right)
&\leq
\int_{\mathbb{R}_{\geq0}\times\mathbb{R}_{\geq0}} W_1^{\leq1}\left(P_t^{\ast}\delta_x,P_t^{\ast}\delta_y\right) H\left(\mathrm{d}x,\mathrm{d}y\right) \\
&\leq
\int_{\mathbb{R}_{\geq0}\times\mathbb{R}_{\geq0}}\mathbb{E}\left[\left\vert X_t^x-X_t^y\right\vert\wedge1\right] H\left(\mathrm{d}x,\mathrm{d}y\right)  \\
&\leq
\int_{\mathbb{R}_{\geq0}\times\mathbb{R}_{\geq0}}\left(1\wedge \mathbb{E}\left[\left\vert X_t^x-X_t^y\right\vert\right]\right) H\left(\mathrm{d}x,\mathrm{d}y\right)
\\ &\leq \int_{\mathbb{R}_{\geq0}\times\mathbb{R}_{\geq0}} \left(1\wedge \left(\vert x-y\vert\mathrm{e}^{-At}\right)\right) H\left(\mathrm{d}x,\mathrm{d}y\right),
\end{align*}
where the last inequality follows from \eqref{eq:moment estimate}.
By dominated convergence we see that the right-hand side vanishes as $t\to\infty$. Consequently, $W_1^{\leq1}(\pi,\widehat{\pi})=0$ which implies that $\pi=\widehat{\pi}$. The proof is complete.
\end{proof}

Combining the characterization of convergence with respect to $W_1$ and
\cite[Theorem 1.2 and Corollary 1.3]{MR3684455}, we deduce the following.

\begin{cor}
Suppose that the same conditions as in Theorem \ref{thm:exponential ergodicity} are satisfied. Let $\lbrace X_t\thinspace:\thinspace t \geq 0\rbrace$ be the corresponding unique solution to \eqref{eq:se}. Then the following assertions hold:
\begin{enumerate}
 \item[(a)] $\lim_{t \to \infty}\mathbb{E}[X_t^x] = \int_{\mathbb{R}_{\geq 0}}x \pi(\mathrm{d}x).$
 \item[(b)] For any $p \in [1,\infty)$ and
 $f \in L^p(\mathbb{R}_{\geq 0}, \pi)$, we have
\[
\frac{1}{T}\int_0^T f(X_t)\mathrm{d}t\to\int_0^{\infty}f(x)\pi\left(\mathrm{d}x\right),\quad T\to\infty,
\]
in $L^p(\Omega,\mathcal{F},\mathbb{P})$.
\end{enumerate}
\end{cor}

\section{Continuous-state nonlinear branching processes with immigration}
\label{sec:CNB}

In this section we constitute an application of Theorem \ref{thm:exponential ergodicity} to the Markov process with generator \eqref{EQ:NONLINEARCBI}.
For this purpose, we first provide a construction of this process as a strong solution to a certain stochastic equation.

Let $(\Omega, \mathcal{F}, (\mathcal{F}_t)_{t \geq 0}, \mathbb{P})$ be a stochastic basis with the usual conditions rich enough to support a $(\mathcal{F}_t)_{t \geq 0}$-Gaussian white noise
$W(\mathrm{d}t,\mathrm{d}u)$ with intensity measure $\mathrm{d}t\mathrm{d}u$ on $\mathbb{R}_{\geq0}\times\mathbb{R}_{\geq0}$ and $(\mathcal{F}_t)_{t \geq 0}$-Poisson random measures
$N_0(\mathrm{d}t,\mathrm{d}z,\mathrm{d}u)$ and $N_1(\mathrm{d}t,\mathrm{d}z)$
with intensity measures $\mathrm{d}tm(\mathrm{d}z)\mathrm{d}u$ on $\mathbb{R}_{\geq0}\times\mathbb{R}_{\geq0}\times\mathbb{R}_{\geq0}$  and $\mathrm{d}t\nu(\mathrm{d}z)$ on $\mathbb{R}_{\geq0}\times\mathbb{R}_{\geq0}$, where $m$ and $\nu$ satisfy \eqref{eq:measure conditions}.
Further, suppose that $W(\mathrm{d}t, \mathrm{d}u)$,
$N_0(\mathrm{d}t, \mathrm{d}z, \mathrm{d}u)$, and $N_1(\mathrm{d}t, \mathrm{d}z)$ are mutually independent.
Denote by
$\widetilde{N}_0(\mathrm{d}t,\mathrm{d}z,\mathrm{d}u)
= N_0(\mathrm{d}t, \mathrm{d}z,\mathrm{d}u)-\mathrm{d}t \mathrm{d}m(\mathrm{d}u)$ the corresponding compensated Poisson random measure.
Below we provide reasonable conditions on $\gamma_0, \thinspace \gamma_1$, and $\gamma_2$
such that
\begin{align}
X_t &=
X_0 + \int_0^t \gamma_0\left(X_s\right)\mathrm{d}s +\int_0^t\int_0^{\infty}\mathbbm{1}_{\lbrace u\leq \gamma_1(X_{s})\rbrace}W\left(\mathrm{d}s,\mathrm{d}u\right) \nonumber \\
&\quad+
\int_0^t\int_0^{\infty}\int_0^{\infty}\mathbbm{1}_{\lbrace u\leq \gamma_2(X_{s-})\rbrace}z\widetilde{N}_0\left(\mathrm{d}s,\mathrm{d}z,\mathrm{d}u\right)
+\int_0^t\int_0^{\infty}zN_1\left(\mathrm{d}s,\mathrm{d}z\right),\label{eq:nonlinear CBI se}
\end{align}
has a pathwise unique strong solution so that the results of Section \ref{sec:Stochastic equations of nonnegative processes with jumps} are applicable.

\begin{thm}\label{thm:pathwise uniqueness strong solution of nonlinear cbi}
Suppose that the functions $\gamma_i$, $i=0,1,2$, satisfy the following:
\begin{enumerate}
\item[(i)] $\gamma_0(0)\geq0$, $\gamma_1,\thinspace\gamma_2\geq0$, and $\gamma_2$ is nondecreasing;
\item[(ii)] there exists a constant $K\geq0$ such that
$\vert\gamma_0(x)\vert\leq K(1+x)$ for all $x \geq 0;$
\item[(iii)] for each $m\geq1$ there exists a constant $c_m>0$ such that,
for all $0 \leq x,y \leq m$,
\[
\left\vert \gamma_0(x)-\gamma_0(y)\right\vert+\left\vert\gamma_1(x)-\gamma_1(y)\right\vert
+\left\vert\gamma_2(x)-\gamma_2(y)\right\vert \leq c_m\vert x-y\vert.
\]
\end{enumerate}
Then, for any $\mathcal{F}_0$-measurable and nonnegative
initial value $X_0$, there exists a unique strong solution $\lbrace X_t\thinspace:\thinspace t\geq0\rbrace$ to \eqref{eq:nonlinear CBI se}.
\end{thm}

\begin{proof}
We are going to apply Proposition \ref{prop:existence and uniqueness of strong solutions} with the following choices:
\begin{itemize}
\item $E= \mathbb{R}_{\geq0}$, $U_0=\mathbb{R}_{\geq0}\times\mathbb{R}_{\geq0}$, $U_1 = \mathbb{R}_{\geq 0}$;
\item $b(x)=b_1(x)=\gamma_0(x)$, $\sigma(x,u)=\mathbbm{1}_{\lbrace u\leq\gamma_1(x)\rbrace}$, $g_0(x,z,u)=z\mathbbm{1}_{\lbrace u\leq\gamma_2(x)\rbrace}$, $g_1(x,z) = z$;
\item $\varkappa(\mathrm{d}u)=\mathrm{d}u$, $\mu_0(\mathrm{d}z,\mathrm{d}u)=m(\mathrm{d}z)\mathrm{d}u$, $\mu_1(\mathrm{d}u)=\nu(\mathrm{d}u)$;
\end{itemize}
Now, it is easy to see that conditions (3.a) and (3.b) are satisfied.
We turn to check condition (3.c). Define $l_0(x,y,u):=\mathbbm{1}_{\lbrace u\leq \gamma_2(x)\rbrace}-\mathbbm{1}_{\lbrace u\leq \gamma_2(y)\rbrace}$. For each $m\geq1$, we estimate
\begin{align*}
\int_0^{\infty} \left\vert \mathbbm{1}_{\lbrace u\leq \gamma_1(x)\rbrace}-\mathbbm{1}_{\lbrace u\leq \gamma_1(y)\rbrace}\right\vert^2\mathrm{d}u
&+
\int_0^{\infty}\int_0^{\infty}\left(\left\vert zl_0(x,y,u)\right\vert \wedge \left\vert zl_0(x,y,u)\right\vert^2\right)m\left(\mathrm{d}z\right)\mathrm{d}u \\
&\leq
\left\vert \gamma_1(x)-\gamma_1(y)\right\vert + |\gamma_2(y)-\gamma_2(x)|  \int_0^{\infty}\left( z\wedge z^2\right)m\left(\mathrm{d}z\right) \\
&\leq c'_m\vert x-y\vert,
\end{align*}
for all $0\leq x$, $y\leq m$ and some constant $c_m' > 0$,
yielding that condition (3.c) holds for $\rho_m(z)=c_m \sqrt{z}$.
\end{proof}

As a consequence of Theorem \ref{thm:pathwise uniqueness strong solution of nonlinear cbi} the unique solution to \eqref{eq:nonlinear CBI se}
is a strong $(\mathcal{F}_t)_{t\geq0}$-Markov process which is called a CNBI process.
Let $\lbrace P_t\thinspace:\thinspace t\geq0\rbrace$ be its transition semigroup and
$\lbrace P_t^*\thinspace:\thinspace t \geq 0\rbrace$ the dual semigroup. Ergodicity of the CNBI process is obtained below.

\begin{thm}\label{thm:exponential ergodicity in wassterstein distance for nonlinear cbi}
Suppose that conditions (i) -- (iii) of Theorem \ref{thm:pathwise uniqueness strong solution of nonlinear cbi} are satisfied and assume the following:
\begin{enumerate}
\item[(a)] there exists a constant $A>0$ such that
\[
\gamma_0(y)-\gamma_0(x)\leq -A(y-x),\quad 0 \leq x \leq y;
\]
\item[(b)] there exists $\lambda \in [1,2]$ and $K > 0$ such that
$\gamma_1(x) + \gamma_2(x) \leq K(1+x)^{\lambda}$, $x \geq 0$, and
\[
\int_{\lbrace z>1\rbrace} z^2 m(\mathrm{d}z) + \int_{\lbrace z>1\rbrace}z^{\lambda} \nu(\mathrm{d}z) < \infty.
\]
\end{enumerate}
Then, for all $\varrho,\thinspace\widetilde{\varrho}\in\mathcal{P}_1(\mathbb{R}_{\geq0})$, we have
\[
W_1\left(P_t^{\ast}\varrho,P_t^{\ast}\widetilde{\varrho}\right)\leq \mathrm{e}^{-At} W_1\left(\varrho,\widetilde{\varrho}\right),\quad t\geq0.
\]
In particular, there exists a unique invariant distribution $\pi\in\mathcal{P}(\mathbb{R}_{\geq0})$.
Moreover, we have $\pi\in\mathcal{P}_1(\mathbb{R}_{\geq0})$ and, for all $\varrho\in\mathcal{P}_1(\mathbb{R}_{\geq0})$,
\begin{align}\label{EQ:04}
W_1\left(P_t^{\ast}\varrho,\pi\right)\leq \mathrm{e}^{-At}W_1\left(\varrho,\pi\right),\quad t\geq0.
\end{align}
\end{thm}

\begin{proof}
In light of Theorem \ref{thm:exponential ergodicity} it only suffices to show that condition (3.d) is satisfied, i.e.
\[
M_t:=
\int_0^t\int_0^{\infty}\mathbbm{1}_{\lbrace u\leq \gamma_1(X_s^x)\rbrace}W(\mathrm{d}s,\mathrm{d}u)
+
\int_0^t\int_0^{\infty}\int_0^{\infty}z\mathbbm{1}_{\lbrace u\leq \gamma_2(X_{s-}^x)\rbrace}\widetilde{N}\left(\mathrm{d}s,\mathrm{d}z,\mathrm{d}u\right),\quad t\geq0
\]
is a martingale with respect to the filtration $(\mathcal{F}_t)_{t\geq0}$. In order to prove that $\lbrace M_t\thinspace:\thinspace t\geq0\rbrace$ is a martingale, it is enough to show that
$\int_{0}^{t}\mathbb{E}\left[ \gamma_1(X_s^x) + \gamma_2(X_s^x) \right] ds < \infty$ which in turn is true if
\begin{align}\label{EQ:01}
\sup_{s \in [0,t]}\mathbb{E}[ (X_s^x)^{\lambda} ] < \infty, \quad t, \thinspace x \geq 0.
\end{align}
Define $V_{\lambda}(x) = (1+x)^{\lambda}$, $x \geq 0$. By It\^o's formula, we obtain
\begin{align}\label{eq:ito formula nonlinear case}
 V_{\lambda}(X_t^x) = V(x) + \int_{0}^{t}(LV_{\lambda})(X_s^x)ds + M_t\left(V_{\lambda}\right),
\end{align}
where, by abuse of notation, we continue to write $LV_{\lambda}$ which is given in \eqref{EQ:NONLINEARCBI}
and $\lbrace M_t(V_{\lambda}) \thinspace:\thinspace t \geq 0\rbrace$ is a local martingale given by
\begin{align*}
M_t\left(V_{\lambda}\right) &=
\int_{0}^{t}\int_{0}^{\infty}V_{\lambda}'(X_s^x) \mathbbm{1}_{ \lbrace u \leq \gamma_1(X_s^x)\rbrace} W(\mathrm{d}s, \mathrm{d}u)\\
&\quad+
\int_{0}^{t}\int_{0}^{\infty} \int_{0}^{\infty}\left( V_{\lambda}\left( X_s^x + z\mathbbm{1}_{\lbrace u\leq \gamma_2(X_{s-}^x)\rbrace}\right) - V_{\lambda}\left(X_s^x\right) \right) \widetilde{N}(\mathrm{d}s, \mathrm{d}z, \mathrm{d}u).
\end{align*}
Define, for $n \in \mathbb{N}$,
a stopping time $\tau_n=\inf\lbrace t\in\mathbb{R}_{\geq0}\thinspace:\thinspace X_t>n\rbrace$.
 It is easy to see that $\lbrace M_{t\wedge \tau_n}(V_{\lambda})\thinspace:\thinspace t\geq0\rbrace$ is a martingale with respect to the filtration $(\mathcal{F}_t)_{t\geq0}$, for any $n\in\mathbb{N}$.
 Hence, taking the expectation in \eqref{eq:ito formula nonlinear case} and
 using Lemma \ref{lem:lyapunov estimate1} from the appendix yields
\[
\mathbb{E}\left[ V_{\lambda}(X_{t\wedge\tau_n}^x)\right]
\leq V_{\lambda}(x) + C\int_{0}^{t} \mathbb{E}\left[ V_{\lambda}(X_{s\wedge\tau_n}^x)\right] \mathrm{d}s,\quad t\geq0,
\]
where $C>0$ is a constant.
By means of Gronwall's lemma, we estimate
$\mathbb{E}\left[V(X_{t\wedge\tau_n}^x)\right] \leq  V_{\lambda}(x)\exp(Ct)$.
Noting that $\lbrace X_t^x\thinspace:\thinspace t\geq0\rbrace$ has c\`adl\`ag paths and $C$ is independent of $n$, we can take the limit $n\to\infty$ and apply the Lemma of Fatou to get $\mathbb{E}[V_{\lambda}(X_t^x)]\leq V_{\lambda}(x)\exp(Ct)$.
This proves \eqref{EQ:01}.
\end{proof}

We end this section by constituting the following example for $\gamma_0,\thinspace \gamma_1,\thinspace \gamma_2$.

\begin{exmp}
Theorem \ref{thm:exponential ergodicity in wassterstein distance for nonlinear cbi} is applicable for the particular choice:
\begin{itemize}
\item $\gamma_0(x)=\beta-bx$ where $\beta\geq0$ and $b>0$ are constants;
\item $\gamma_1(x)=x^{\alpha}$ with $\alpha\in[1,2]$;
\item $\gamma_2(x)=x^{\delta}$ with $\delta\in[1,2]$.
\end{itemize}
\end{exmp}

The reader may wonder why we have apparently excluded the cases when $\alpha$ and $\delta$ either belong to $(0,1)$ or $(2,\infty)$. If $\alpha,\thinspace \delta \in (0,1)$, then we may loose uniqueness, while for $\alpha,\thinspace \delta \in (2,\infty)$
the corresponding process may have an explosion.
In both cases we may try, following \cite{2017arXiv170801560L},
to study solutions having a trap at $0$ and/or at $\infty$.
However, since in this case $\delta_0$ and $\delta_{\infty}$ would be invariant distributions, \eqref{EQ:04} cannot hold in general.

\section{Continuous-state branching processes with immigration}
\label{sec:CBI}


Recall that $(\beta, b, \sigma, m, \nu)$ are admissible parameters, if
$\beta\geq0$, $b\in\mathbb{R}$, and $\sigma\geq0$ are constants, and $m(\mathrm{d}z)$ and $\nu(\mathrm{d}z)$ are L\'evy-measures on $\mathbb{R}_{\geq0}$ satisfying \eqref{eq:measure conditions}.
For $\lambda\geq0$, define the \emph{branching mechanism}
\[
\phi(\lambda)=b\lambda+\frac{1}{2}\sigma^2\lambda^2+
\int_0^{\infty}\left(\mathrm{e}^{-\lambda z}-1+\lambda z\right)m\left(\mathrm{d}z\right),
\]
and \emph{immigration mechanism}
\[
\psi(\lambda)=\beta\lambda+\int_0^{\infty}\left(1-\mathrm{e}^{-\lambda z}\right)\nu\left(\mathrm{d}z\right).
\]
A Markov process with state space $\mathbb{R}_{\geq0}$ is called a CBI process with admissible parameters $(\beta, b, \sigma, m, \nu)$
if its transition semigroup $\lbrace P_t\thinspace:\thinspace t\geq0\rbrace$ has the representation
\begin{equation}
\int_0^{\infty} \mathrm{e}^{-\lambda z}P_t\left(x,\mathrm{d}z\right)
=
\exp\left(-xv_t(\lambda)-\int_0^t\psi\left(v_s(\lambda)\right)\mathrm{d}s\right),\quad x,\thinspace\lambda\geq0,\label{eq:transition semigroup cbi}
\end{equation}
where $t\mapsto v_t(\lambda)$ is the unique nonnegative solution of the ODE
\begin{equation}
\frac{\partial}{\partial t}v_t(\lambda)=-\phi\left(v_t\left(\lambda\right)\right),\quad v_0\left(\lambda\right)=\lambda.\label{eq:def of v_t}
\end{equation}
It can be shown that the corresponding Markov process is a Feller process, $C_c^{\infty}(\mathbb{R}_{\geq 0})$ is a core for its generator, and
the action of the generator is given by \eqref{eq:generator cbi}.
Moreover, it is well-known that each CBI process can be obtained as the unique strong solution
to a certain stochastic equation with jumps, see, e.g., \cite{MR2243880,MR2584896}.
In what follows, we briefly describe this equation.

Let $(\Omega, \mathcal{F}, (\mathcal{F}_t)_{t \geq 0}, \mathbb{P})$ be a stochastic basis satisfying the usual conditions rich enough to support the following random objects:
a $(\mathcal{F}_t)_{t \geq 0}$-Gaussian white noise
$W(\mathrm{d}s,\mathrm{d}u)$ with intensity $\mathrm{d}t\mathrm{d}u$, $(\mathcal{F}_t)_{t \geq 0}$-Poisson random measures
$N_0(\mathrm{d}t,\mathrm{d}z,\mathrm{d}u)$
and $N_1(\mathrm{d}t,\mathrm{d}z)$ on $\mathbb{R}_{\geq0}\times\mathbb{R}_{\geq0}\times\mathbb{R}_{\geq0}$ and $\mathbb{R}_{\geq0}\times\mathbb{R}_{\geq0}$ with intensities $\mathrm{d}tm(\mathrm{d}z)\mathrm{d}u$ and $\mathrm{d}t\nu(\mathrm{d}z)$, respectively. We denote by $\widetilde{N}_0(\mathrm{d}t,\mathrm{d}z,\mathrm{d}u)$ the compensated measure of $N_0(\mathrm{d}t,\mathrm{d}z,\mathrm{d}u)$. Suppose that $W(\mathrm{d}s,\mathrm{d}u)$, $N_0(\mathrm{d}s,\mathrm{d}z,\mathrm{d}u)$, and $ N_1(\mathrm{d}s,\mathrm{d}z)$ are mutually independent.
Then, for each $\mathcal{F}_0$-measurable $X_0 \geq 0$,
there exists a unique strong solution to
\begin{align}
X_t
&=
X_0 + \int_0^t \left(\beta-bX_s\right)\mathrm{d}s+\sigma\int_0^t\int_0^{\infty}\mathbbm{1}_{\lbrace u\leq X_{s}\rbrace} W\left(\mathrm{d}s,\mathrm{d}u\right) \nonumber \\
&\quad+
\int_0^t\int_0^{\infty}\int_0^{\infty}z\mathbbm{1}_{\lbrace u\leq X_{s-}\rbrace}\widetilde{N}_0\left(\mathrm{d}s,\mathrm{d}z,\mathrm{d}u\right)+
\int_0^t\int_0^{\infty}zN_1\left(\mathrm{d}s,\mathrm{d}z\right).\label{eq:se cbi}
\end{align}
In order to deduce this result one may, e.g.,
apply Proposition \ref{prop:existence and uniqueness of strong solutions}.
It\^o's formula shows that $\lbrace X_t \thinspace:\thinspace t\geq0\rbrace$ is a Markov process whose generator is given by \eqref{eq:generator cbi}.
Conversely, the law of any CBI process with admissible parameters $(\beta, b, \sigma, m ,\nu)$ can be obtained from \eqref{eq:se cbi}.

The following is a particular case of \cite[Lemma 3.4]{MR3340375}.

\begin{rem}
Let $(\beta, b, \sigma, m ,\nu)$ be admissible parameters with $\nu$ satisfying
\begin{equation}
\int_{\{ z > 1\}}z\nu\left(\mathrm{d}z\right)<\infty.\label{eq:integrability condition}
\end{equation}
Then $X_t^x$ has finite first moment given by
\[
\mathbb{E}\left[X_t^x\right]=\mathrm{e}^{-bt}x+\left(\beta+\int_{\mathbb{R}_{\geq 0}}z\nu\left(\mathrm{d}z\right)\right)
\int_0^t\mathrm{e}^{-bs}\mathrm{d}s,\quad t\geq0.
\]
\end{rem}

\subsection{Exponential ergodicity in Wasserstein distance}
\label{subsec:exp ergo in wasserstein distances cbi}

The aim of this subsection is to derive exponential ergodicity in the Wasserstein distance $W_{\log}$ for CBI processes. Let us start with a general characterization of the existence and uniqueness
of the invariant distribution. The next theorem is proved in \cite{MR2760602}.

\begin{thm}[Li \cite{MR2760602}, Theorem 3.20] \label{rem:invariant measure for cbi}
Let $(\beta, b, \sigma, m, \nu)$ be admissible
parameters satisfying $b\geq0$ and $\phi(\lambda)\not=0$ for $\lambda>0$.
Then the following are equivalent:
\begin{enumerate}
 \item[(a)] There exists $x \geq 0$ and $\pi \in \mathcal{P}(\mathbb{R}_{\geq 0})$ such that $P_t(x,\cdot) \to \pi$ weakly as $t \to \infty$.
 \item[(b)] The branching and immigration mechanisms satisfy
 \begin{equation}
\int_0^{\lambda}\frac{\psi(u)}{\phi(u)}\mathrm{d}u<\infty\quad\text{for some }\lambda>0.\label{eq:existence of inv measure for cbi condition}
 \end{equation}
\end{enumerate}
Moreover, if either (a) or (b) is satisfied, then $P_t(x,\cdot)$ converges weakly to $\pi$ as $t\to\infty$ for all $x\geq0$, $\pi$ is the unique invariant distribution and its Laplace transform is given by
\begin{equation}
\int_0^{\infty}\mathrm{e}^{-\lambda x}\pi\left(\mathrm{d}x\right)
= \exp\left(-\int_0^{\lambda}\frac{\psi(u)}{\phi(u)}\mathrm{d}u\right),\quad \lambda\geq0.\label{eq:laplace transform of pi}
\end{equation}
\end{thm}

The following is due to \cite[Corollary 3.21]{MR2760602}.

\begin{rem}
If $b > 0$, then \eqref{eq:existence of inv measure for cbi condition} is equivalent to
\eqref{eq:inequality w_log pi}.
\end{rem}

As a consequence of Theorem \ref{thm:exponential ergodicity},
we find the following ergodicity result for CBI processes in the Wasserstein distance $W_1$.

\begin{rem}\label{rem:exp ego in wasserstein distence w_1 cbi}
Let $(\beta, b, \sigma, m ,\nu)$ be admissible parameters
with $b > 0$ and $\nu$ satisfying \eqref{eq:integrability condition}.
Then all assumptions of Theorem \ref{thm:exponential ergodicity} are satisfied
implying that the unique invariant distribution $\pi$ given by Theorem \ref{rem:invariant measure for cbi} has finite first moment and satisfies, for all $\varrho\in\mathcal{P}_1(\mathbb{R}_{\geq0})$,
\[
W_1\left(P_t^{\ast}\varrho,\pi\right)\leq \mathrm{e}^{-bt}W_1\left(\varrho,\pi\right),\quad t\geq0.
\]
\end{rem}

An analogue result for affine processes was recently established in \cite{FJR}. Note that the class of affine processes include the (multidimensional) CBI processes as a special case.
Moreover, it was shown in \cite{FJR} that under the analogue of \eqref{eq:inequality w_log pi}
the statement of Remark \ref{rem:exp ego in wasserstein distence w_1 cbi} holds true for affine processes even in the Wasserstein distance $W_{\log}$.
Motivated by the fact that our main result (see Theorem \ref{thm:exponential ergodicity cbi} below)
requires the existence of $\log$-moments for the invariant distribution $\pi$, we recall the result of \cite[Theorem 1.6 (a)]{FJR} for our purpose.

\begin{thm}\label{thm:exponential ergodicity in wasserstein distance of cbi processes}
Let $(\beta, b, \sigma, m, \nu)$ be admissible parameters
and assume that $b > 0$ and \eqref{eq:inequality w_log pi} are satisfied.
Let $\lbrace P_t\thinspace:\thinspace t \geq 0\rbrace$ be the corresponding transition semigroup.
Then there exists a constant $C > 0$ such that,
for all $\varrho,\thinspace\widetilde{\varrho}\in\mathcal{P}_{\log}(\mathbb{R}_{\geq0})$, we have
\begin{equation}
W_{\log}\left(P_t^{\ast}\varrho,P_t^{\ast}\widetilde{\varrho}\right)\leq C\min\left\lbrace \mathrm{e}^{-bt}, W_{\log}\left(\varrho,\widetilde{\varrho}\right)\right\rbrace+C\mathrm{e}^{-bt}W_{\log}\left(\varrho,\widetilde{\varrho}\right),\quad t\geq0.\label{eq:exp ergodicity in wasserstein distance under log condition}
\end{equation}
In particular, the unique invariant distribution $\pi$ belongs to $\mathcal{P}_{\log}(\mathbb{R}_{\geq0})$ and satisfies
\[
W_{\log}\left(P_t^{\ast}\varrho,\pi\right)\leq C \min\left\lbrace \mathrm{e}^{-bt}, W_{\log}\left(\varrho,\pi\right) \right\rbrace + C\mathrm{e}^{-bt}W_{\log}\left(\varrho,\pi\right),\quad t\geq0.
\]
\end{thm}

In contrast to \cite[Theorem 1.6 (a)]{FJR}, the proof can be significantly simplified in our context, i.e., when dealing with one-dimensional CBI processes.
For convenience of the reader and in order to keep this work self-contained, we provide a sketch of the proof.

\begin{proof}[Proof of Theorem \ref{thm:exponential ergodicity in wasserstein distance of cbi processes}]
We divide the proof in three steps.

\textit{Step 1:} We show that $P_t^*(\mathcal{P}_{\log}(\mathbb{R}_{\geq 0}))\subset \mathcal{P}_{\log}(\mathbb{R}_{\geq 0})$, $t \geq 0$.
Let $V(x):=\log(1+x)$, $x\geq0$.
Then it suffices to find a constant $C>0$ such that
\begin{equation}
\mathbb{E}\left[V\left(X_t\right)\right]\leq Ct+ \mathbb{E}[V(X_0)],\quad t\geq0.\label{eq:moment inequality for v}
\end{equation}
We apply It\^o's formula to $V(x)$, and get
\begin{equation}
V\left(X_t\right) = V\left(X_0\right) + \int_0^tLV\left(X_s\right)\mathrm{d}s+M_t(V),\quad t\geq0,\label{eq:ito formula to v}
\end{equation}
where $LV$ is informally defined by \eqref{eq:generator cbi} and
\begin{align*}
M_t(V)&:=\sigma\int_0^t\int_0^{\infty}\mathbbm{1}_{\lbrace u\leq X_s\rbrace}W\left(\mathrm{d}s,\mathrm{d}u\right) \\
&\thinspace\quad+
\int_0^t\int_0^{\infty}\int_0^{\infty}\left(V\left(X_{s-}+z\mathbbm{1}_{\lbrace u\leq X_{s-}\rbrace}\right)-V\left(X_{s-}\right)\right)\widetilde{N}_0\left(\mathrm{d}s,\mathrm{d}z,\mathrm{d}u\right) \\
&\thinspace\quad
+\int_0^t\int_0^{\infty}\left(V\left(X_{s-}+z\right)-V\left(X_{s-}\right)\right)\widetilde{N}_1\left(\mathrm{d}s,\mathrm{d}z\right),\quad t\geq0,
\end{align*}
where $\widetilde{N}_1(\mathrm{d}s,\mathrm{d}z)$ denotes the compensated $(\mathcal{F}_t)_{t\geq0}$-Poisson random measure of $N_1(\mathrm{d}s,\mathrm{d}z)$. For $n\in\mathbb{N}$, define a stopping time by $\tau_n=\inf\lbrace t\in\mathbb{R}_{\geq0} \thinspace : \thinspace X_t>n\rbrace$.
Then $\lbrace M_{t\wedge\tau_n}\thinspace:\thinspace t\geq0\rbrace$ is a martingale with respect to the filtration $(\mathcal{F}_t)_{t\geq0}$ for any $n\in\mathbb{N}$. Hence, taking expectations in \eqref{eq:ito formula to v} and using that $LV\leq C$, for all $x\geq0$ and a constant $C>0$, see Lemma \ref{lem:lyapunov estimate} in the appendix, gives
\begin{align*}
\mathbb{E}\left[V\left(X_{t\wedge\tau_n}\right)\right] &=
 \mathbb{E}[V(X_0)] + \mathbb{E}\left[\int_0^tLV\left(X_{s\wedge\tau_n}\right)\mathrm{d}s\right]
 \leq  \mathbb{E}[V(X_0)] + Ct.
\end{align*}
Noting that $\lbrace X_t\thinspace:\thinspace t\geq0\rbrace$ has c\`adl\`ag paths and $C$ is independent of $n$, we can take the limit $n\to\infty$ and apply Fatou's lemma to conclude with \eqref{eq:moment inequality for v}.

\textit{Step 2:} From now on, we can proceed very close to the proof of Theorem \ref{thm:exponential ergodicity}, albeit with some slightly different estimates. The details are as follows: let $\lbrace Q_t\thinspace:\thinspace t\geq0\rbrace$ be the transition semigroup with admissible parameters $b>0$, $\beta=0$, $m$, and $\nu\equiv0$. Hence, for $\lambda\geq0$, we have
\[
\int_0^{\infty}\mathrm{e}^{-\lambda z}Q_t\left(x,\mathrm{d}z\right)=\mathrm{e}^{-xv_t(\lambda)}\quad\text{and}\quad
\int_0^{\infty}\mathrm{e}^{-\lambda z}P_t\left(0,\mathrm{d}z\right)= \exp\left(-\int_0^t\psi\left(v_s(\lambda)\right)\mathrm{d}s\right),
\]
and consequently, $P_t(x,\cdot)=Q_t(x,\cdot)\ast P_t(0,\cdot)$.
Noting that $Q_t$ satisfies the conditions of
Theorem \ref{thm:exponential ergodicity}, we obtain from Step 1 and 2
of its proof
\[
W_1\left(Q_t^{\ast}\delta_x,Q_t^{\ast}\delta_y\right)\leq \mathrm{e}^{-bt} \vert x-y\vert,\quad t\geq0,
\]
where $\lbrace Q_t^{\ast}\thinspace:\thinspace t\geq0\rbrace$ denotes the dual semigroup of $Q_t$.
With this, and applying Lemma \ref{lem:convolution estimate} (a), we get
\begin{align}
W_{\log}\left(P_t^{\ast}\delta_x,P_t^{\ast}\delta_y\right)
&\leq W_{\log}\left(Q_t^{\ast}\delta_x,Q_t^{\ast}\delta_y\right) \nonumber \\
&\leq \log\left(1+W_1\left(Q_t^{\ast}\delta_x,Q_t^{\ast}\delta_y\right)\right)
\leq \log\left(1+ \mathrm{e}^{-bt} \vert x-y\vert\right).\label{eq:00}
\end{align}
Moreover, for $a,\thinspace d\geq0$, we use the elementary inequality
\begin{align}
\log(1+a\cdot d) \leq C \min\{a, \log(1+d)\} + Ca\log(1+d) \label{eq:elementary log inequality},
\end{align}
where $C > 0$ is a generic constant, to obtain from \eqref{eq:00}
\[
W_{\log}\left(P_t^{\ast}\delta_x,P_t^{\ast}\delta_y\right)
\leq C\min\{e^{-bt}, \log\left(1+\vert x-y\vert\right)\}+Ce^{-bt}\log\left(1+\vert x-y\vert\right).
\]
By using the convexity of $W_{\log}$ we get, for any $H\in\mathcal{H}(\varrho,\widetilde{\varrho})$,
\begin{align*}
W_{\log}\left(P_t^{\ast}\varrho,P_t^{\ast}\widetilde{\varrho}\right)
&\leq \int_{\mathbb{R}_{\geq0}\times\mathbb{R}_{\geq0}}W_{\log}\left(P_t^{\ast}\delta_x,P_t^{\ast}\delta_y\right)
H\left(\mathrm{d}x,\mathrm{d}y\right) \\
&\leq C\int_{\mathbb{R}_{\geq0}\times\mathbb{R}_{\geq0}}\min\left\lbrace e^{-bt},  \log\left(1+\vert x-y\vert\right)\right\rbrace H\left(\mathrm{d}x,\mathrm{d}y\right) \\
&\quad+ Ce^{-bt}\int_{\mathbb{R}_{\geq0}\times\mathbb{R}_{\geq0}}\log\left(1+\vert x-y\vert\right)H\left(\mathrm{d}x,\mathrm{d}y\right) \\
&\leq C \min\left\{ \mathrm{e}^{-bt}, \int_{\mathbb{R}_{\geq0}\times\mathbb{R}_{\geq0}}\log\left(1+\vert x-y\vert\right)H\left(\mathrm{d}x,\mathrm{d}y\right) \right\} \\
&\quad+ C e^{-bt}\int_{\mathbb{R}_{\geq0}\times\mathbb{R}_{\geq0}}\log\left(1+\vert x-y\vert\right)H\left(\mathrm{d}x,\mathrm{d}y\right).
\end{align*}
Finally, taking $H$ as the optimal coupling of $(\varrho,\widetilde{\varrho})$, we deduce that
\[
W_{\log}\left(P_t^{\ast}\varrho,P_t^{\ast}\widetilde{\varrho}\right)
\leq C \min\{\mathrm{e}^{-bt}, W_{\log}\left(\varrho,\widetilde{\varrho}\right)\} + C\mathrm{e}^{-bt}W_{\log}\left(\varrho,\widetilde{\varrho}\right).
\]

\textit{Step 3:} Let $\varrho \in \mathcal{P}_{\log}(\mathbb{R}_{\geq 0})$.
Arguing similar to the proof of Theorem \ref{thm:exponential ergodicity},
we see that $(P_t^{\ast}\varrho)_{k \in \mathbb{N}} \subset\mathcal{P}_{\log}(\mathbb{R}_{\geq0})$ is a Cauchy sequence and, thus, has a limit $\widehat{\pi}\in\mathcal{P}_{\log}(\mathbb{R}_{\geq0})$. Clearly, $\widehat{\pi}$ is an invariant distribution and hence, by uniqueness of the invariant distribution, $\pi=\widehat{\pi}\in\mathcal{P}_{\log}(\mathbb{R}_{\geq0})$. Eventually, noting that
\[
W_{\log}\left(P_t^{\ast}\varrho,\pi\right)=W_{\log}\left( P_t^{\ast}\varrho,P_t^{\ast}\pi\right)
\leq C \min\left\lbrace\mathrm{e}^{-bt},W_{\log}\left(\varrho,\pi\right)\right\rbrace+ C\mathrm{e}^{-bt}W_{\log}\left(\varrho,\pi\right),
\]
we conclude with \eqref{eq:exp ergodicity in wasserstein distance under log condition}.
This completes the proof.
\end{proof}

\subsection{Exponential ergodicity in total variation distance}
\label{subsec:exp ergo in tv cbi}

A general exponential ergodicity result in the total variation distance
was recently obtained by Li and Ma \cite{MR3343292}, where Grey's condition was used:
\begin{enumerate}
 \item[(5.a)] there exists $\theta>0$ such that $\phi(\lambda)>0$ for $\lambda>\theta$ and
\[
\int_{\theta}^{\infty}\phi(\lambda)^{-1}\mathrm{d}\lambda<\infty.
\]
\end{enumerate}

The following was shown in \cite{MR3343292}.

\begin{thm}[\cite{MR3343292}]\label{prop:exponential ergodicity of P_t^0}
Let $(\beta, b, \sigma, m, \nu)$ be admissible parameters.
Suppose that $b > 0$, $\nu\equiv0$, and (5.a) is satisfied. Let $\lbrace P_t^0\thinspace:\thinspace t \geq 0\rbrace$ be the corresponding transition semigroup given by
\begin{equation}
\int_{\mathbb{R}_{\geq0}}\mathrm{e}^{-\lambda z}P_t^0\left(x,\mathrm{d}z\right)
=
\exp\left(-xv_t\left(\lambda\right)-\beta\int_0^tv_s\left(\lambda\right)\mathrm{d}s\right),\quad \lambda\geq0,\label{eq:transition semigroup P_t^0}
\end{equation}
where $v_t$ is determined by \eqref{eq:def of v_t}.
Denote by $\pi^0$ the corresponding unique invariant distribution given by Theorem \ref{rem:invariant measure for cbi}.
Then there exists a constant $C > 0$ such that,
for all $t,\thinspace x,\thinspace y \geq 0$,
\[
\left\Vert P_t^0(x,\cdot)-P_t^0(y,\cdot)\right\Vert_{TV}
\leq C \min \left\{1, \mathrm{e}^{-bt}\vert x-y\vert \right\}
\]
and
\[
\left\Vert P_t^0(x,\cdot)-\pi^0(\cdot)\right\Vert_{TV}
\leq C \min\left\{ 1, \left(x+\beta b^{-1}\right)\mathrm{e}^{-bt}\right\}
\]
are satisfied. In particular, $\lbrace P_t^0\thinspace:\thinspace t \geq 0\rbrace$ has the strong Feller property and is exponentially ergodic in the total variation distance.
\end{thm}

Our main theorem extends the result of Li and Ma \cite{MR3343292} to CBI processes with non-vanishing jump measure $\nu$ for immigration.

\begin{thm}\label{thm:exponential ergodicity cbi}
Let $(\beta,b,\sigma,m,\nu)$ be admissible parameters with $b > 0$ and suppose that \eqref{eq:inequality w_log pi} and (5.a) are satisfied.
Let $\lbrace P_t\thinspace:\thinspace t\geq0\rbrace$ be the transition semigroup given by \eqref{eq:transition semigroup cbi} and let $\pi$ be the unique invariant distribution. Then the following holds:
\begin{enumerate}
\item[(a)] There exists a constant $C > 0$ such that
\[
\left\Vert P_t(x,\cdot)-P_t(y,\cdot)\right\Vert_{TV}\leq C \min\left\{ 1, \mathrm{e}^{-bt}\vert x-y\vert\right\}.
\]
In particular, $\lbrace P_t\thinspace:\thinspace t \geq 0\rbrace$ has the strong Feller property.
\item[(b)] There exists a constant $C > 0$ such that, for all $\varrho\in\mathcal{P}_{\log}(\mathbb{R}_{\geq0})$ and $t\geq0$, we have
\[
\left\Vert P_t^{\ast}\varrho-\pi\right\Vert_{TV}
\leq
C \min\{e^{-bt}, W_{\log}\left(\varrho,\pi\right)\} + C \mathrm{e}^{-bt}W_{\log}\left(\varrho,\pi\right).
\]
\end{enumerate}
\end{thm}

\begin{proof} (a) Let $\lbrace P_t^0\thinspace:\thinspace t\geq0\rbrace$ be given by \eqref{eq:transition semigroup P_t^0} and let $\lbrace P_t^1\thinspace:\thinspace t \geq 0\rbrace$ be given by
\[
\int_{\mathbb{R}_{\geq0}}\mathrm{e}^{-\lambda z}P_t^1\left(x,\mathrm{d}z\right)
=
\exp\left(-\int_0^t\int_0^{\infty}\left(1-\mathrm{e}^{-v_s(\lambda)z}\right)\nu\left(\mathrm{d}z\right)\mathrm{d}s\right),\quad \lambda\geq0,
\]
where in both cases $v_t$ is obtained from \eqref{eq:def of v_t}.
By definition of the immigration mechanism we have, for all $\lambda \geq 0$,
\begin{align*}
\int_{\mathbb{R}_{\geq0}}\mathrm{e}^{-\lambda z}P_t\left(x,\mathrm{d}z\right)
&=
\int_{\mathbb{R}_{\geq0}}\mathrm{e}^{-\lambda z}P_t^0\left(x,\mathrm{d}z\right)
\int_{\mathbb{R}_{\geq0}}\mathrm{e}^{-\lambda z}P_t^1\left(x,\mathrm{d}z\right),
\end{align*}
yielding that $P_t(x,\cdot)=P_t^0(x,\cdot)\ast P_t^1(0,\cdot)$ for
all $t,\thinspace x \geq 0$. Combining the latter with Lemma \ref{lem:convolution estimate} (b), we deduce
\begin{align*}
\left\Vert P_t(x,\cdot) - P_t(x,\cdot)\right\Vert_{TV}
\leq
\left\Vert P_t^0(x,\cdot) - P_t^0(x,\cdot) \right\Vert_{TV}
\leq C \min\left\{ 1, \mathrm{e}^{-bt}\vert x-y\vert\right\},
\end{align*}
where the last inequality follows from Theorem \ref{prop:exponential ergodicity of P_t^0}.

(b) From Theorem \ref{thm:exponential ergodicity in wasserstein distance of cbi processes} we know that $\pi\in\mathcal{P}_{\log}(\mathbb{R}_{\geq0})$. Let $\varrho\in\mathcal{P}_{\log}(\mathbb{R}_{\geq0})$ and $H\in\mathcal{H}(\varrho,\pi)$ such that the infimum is attained
(see \eqref{EQ:INFATTAINED}).
Here and below we let $C > 0$ be a generic constant which may vary from line to line.
Using the invariance of $\pi$ combined with the convexity of the Wasserstein distance (see Lemma \ref{lem:convexity of wasserstein distances}), shows that
\begin{align*}
\left\Vert P_t^{\ast}\varrho-\pi\right\Vert_{TV}
&\leq
\int_{\mathbb{R}_{\geq0}\times\mathbb{R}_{\geq0}} \left\Vert P_t(x,\cdot)-P_t(y,\cdot)\right\Vert_{TV}H\left(\mathrm{d}x,\mathrm{d}y\right)
\\ &\leq
 C\int_{\mathbb{R}_{\geq0}\times\mathbb{R}_{\geq0}}\log\left(1+ \mathrm{e}^{-b t}\vert x-y\vert\right)H\left(\mathrm{d}x,\mathrm{d}y\right),
\end{align*}
where the last inequality follows from statement (a)
and $1\wedge a\leq \log(2)^{-1}\log(1+a)$ for all $a\geq0$.
Finally, using \eqref{eq:elementary log inequality} and the same estimates as in Step 2 of the proof of Theorem \ref{thm:exponential ergodicity in wasserstein distance of cbi processes}, we readily deduce that
\begin{align*}
\left\Vert P_t^{\ast}\varrho-\pi\right\Vert_{TV}
&\leq
C \min\left\lbrace \mathrm{e}^{-bt}, W_{\log}(\varrho, \pi)\right\rbrace
+ C e^{-bt} W_{\log}(\varrho, \pi).
\end{align*}
\end{proof}

The following remark shows that the obtained convergence has indeed exponential rate.

\begin{rem}\label{rem: tv standard presentation}
Under the assumptions of Theorem \ref{thm:exponential ergodicity cbi}, we obtain for all $x, \thinspace t \geq 0$
\begin{align*}
\left\Vert P_t(x,\cdot)-\pi(\cdot)\right\Vert_{TV}
&\leq
C\mathrm{e}^{-bt}\left(1+W_{\log}\left(\delta_x,\pi\right)\right)\\
&\leq
C\mathrm{e}^{-bt}\left(\log(1+x)+\int_{\mathbb{R}_{\geq0}}\log(1+y)\pi\left(\mathrm{d}y\right)\right).
\end{align*}
In particular, $\lbrace P_t\thinspace:\thinspace t\geq0\rbrace$ is exponentially ergodic in the total variation distance.
\end{rem}

\subsection{Functional central limit theorem}
\label{subsec:fclt cbi}

A direct consequence of our ergodic result is the following strong law of large numbers in accordance with the discussion after \cite[Proposition 2.5]{MR663900}.

\begin{cor}\label{cor:slln}
Under the conditions of Theorem \ref{thm:exponential ergodicity cbi},
for all Borel functions $f:\mathbb{R}_{\geq0}\to\mathbb{R}$ with $\int_0^{\infty}\vert f(x)\vert\pi(\mathrm{d}x)<\infty$, it holds
\[
\frac{1}{t}\int_0^{t}f(X_s^x)\mathrm{d}s\to\int_0^{\infty}f(x)\pi\left(\mathrm{d}x\right) \quad \text{a.s. as }t\to\infty.
\]
\end{cor}

The latter convergence may be very useful for parameter estimation of one-dimensional CBI processes. A further consequence of our ergodicity result is the functional central limit theorem which is stated below.

Recall that the Feller semigroup $\lbrace P_t\thinspace:\thinspace t\geq0\rbrace$ given by \eqref{eq:transition semigroup cbi} has infinitesimal generator $(L,\mathrm{dom}(L))$ of the form \eqref{eq:generator cbi} acting on $C_c^2(\mathbb{R}_{\geq0})$. By virtue of \cite[Theorem 2.7]{MR1994043}, $C_c^{\infty}(\mathbb{R}_{\geq0})$ is a core of $L$ and $C_c^2(\mathbb{R}_{\geq0})\subset\mathrm{dom}(L)$. Since
$\Vert f\Vert_{L^2(\mathbb{R}_{\geq0},\pi)}\leq \Vert f\Vert_{\infty}$ for $f\in C_0(\mathbb{R}_{\geq0})$, $C_0(\mathbb{R}_{\geq0})\subset L^2(\mathbb{R}_{\geq0},\pi)$ is dense, and $P_t$ satisfies $\Vert P_tf\Vert_{L^2(\mathbb{R}_{\geq0},\pi)}\leq \Vert f\Vert_{L^2(\mathbb{R}_{\geq0},\pi)}$ for $f\in C_0(\mathbb{R}_{\geq0})$, there exists a unique extension $\lbrace\widehat{P}_t\thinspace:\thinspace t\geq0\rbrace$ on $L^2(\mathbb{R}_{\geq0},\pi)$. This extension is again a strongly continuous semigroup. Let $(\widehat{L},\mathrm{dom}(\widehat{L}))$ be its infinitesimal generator. Then $\mathrm{dom}(L)\subset\mathrm{dom}(\widehat{L})$ and $Lf=\widehat{L}f$ for all $f\in\mathrm{dom}(L)$.
Define $\mathrm{range}(\widehat{L})=\lbrace \widehat{L}f\thinspace:\thinspace f\in\mathrm{dom}(\widehat{L})\rbrace$. The next result is the announced functional central limit theorem for CBI processes.

\begin{cor}\label{cor:FCLT}
Under the conditions of Theorem \ref{thm:exponential ergodicity cbi}, for all $f\in\mathrm{range}(\widehat{L})$, it holds
\[
n^{-1/2}\int_0^{nt}f\left(X_s^x\right)\mathrm{d}s\to W_{\cdot}\quad\text{ weakly as } n\to\infty,
\]
where $W$ is a one-dimensional Wiener process with zero drift and variance parameter $\gamma$ given by
\begin{align}\label{eq:gamma}
\gamma^2=-2\int_{\mathbb{R}_{\geq0}}\widehat{L}f(y)\cdot f(y)\pi\left(\mathrm{d}y\right).
\end{align}
\end{cor}

In general it is unlikely to find an explicit formula for \eqref{eq:gamma} in terms
of its admissible parameters. However, for the particular case $Lf_{\lambda} \in \mathrm{range}(\widehat{L})$, where $f_{\lambda}(y)=\exp(-\lambda y)$, we obtain the following.

\begin{exmp}
For each $\lambda > 0$ and $x \geq 0$ we see that $n^{-1/2}\int_0^{nt}(\widehat{L}f_{\lambda})(X_s^x)\mathrm{d}s$ converges weakly as $n\to\infty$ to a Wiener process with zero drift and variance parameter $\gamma$ given by
\[
\gamma^2=-2\int_{\mathbb{R}_{\geq0}}\widehat{L}f_{\lambda}(y)\cdot f_{\lambda}(y)\pi\left(\mathrm{d}y\right).
\]
Furthermore, an easy calculation shows that
$\widehat{L}f_{\lambda}(y)=Lf_{\lambda}(y)=f_{\lambda}(y)(-\psi(\lambda)+y\phi(\lambda))$ and thereby
\[
\gamma^2=2\psi(\lambda)\int_{\mathbb{R}_{\geq0}}\mathrm{e}^{-2\lambda y}\pi\left(\mathrm{d}y\right)-2\phi(\lambda)\int_{\mathbb{R}_{\geq0}}y\mathrm{e}^{-2\lambda y}\pi\left(\mathrm{d}y\right).
\]
Recall that the Laplace transform of $\pi$ is given in \eqref{eq:laplace transform of pi}. By a change of variables $u=v_s(\lambda)$, we see that
\[
\int_0^{\infty} \mathrm{e}^{-\lambda y}\pi\left(\mathrm{d}y\right)=\exp\left(-\int_0^{\lambda}\frac{\psi(u)}{\phi(u)}\mathrm{d}u\right),\quad u\in\mathbb{R}_{\geq0},
\]
see also \cite[Formula (3.27)]{MR2390186}. Therefore, we have
\begin{align*}
\gamma^2
&=
2\psi(\lambda)\exp\left(-\int_0^{2\lambda}\frac{\psi(u)}{\phi(u)}\mathrm{d}u\right)
+\phi(\lambda)\frac{\mathrm{d}}{\mathrm{d}\lambda}\exp\left(-\int_0^{2\lambda}\frac{\psi(u)}{\phi(u)}\mathrm{d}u\right)\\
&=
\left(2\psi(\lambda)+\frac{\phi(\lambda)\psi(2\lambda)}{\phi(2\lambda)}\right)
\exp\left(-\int_0^{2\lambda}\frac{\psi(u)}{\phi(u)}\mathrm{d}u\right).
\end{align*}
\end{exmp}

\section{Continuous-state branching processes with immigration in L\'evy random environments}
\label{sec:CBIRE}

Let $(\beta, b, \sigma, m, \nu)$ be admissible parameters,
 $b_{E} \in \mathbb{R}$, $\sigma_{E} \geq 0$ and $\mu_E$ a L\'evy measure on $\mathbb{R}$.
We start with a brief description of CBIRE processes.
Let $(\Omega, \mathcal{F}, (\mathcal{F}_t)_{t \geq 0}, \mathbb{P})$ be a stochastic basis satisfying the usual conditions rich enough to support
\begin{itemize}
\item a $(\mathcal{F}_t)_{t \geq 0}$-Gaussian white noise
$W(\mathrm{d}t,\mathrm{d}u)$ with intensity $\mathrm{d}t\mathrm{d}u$;
\item a $(\mathcal{F}_t)_{t \geq 0}$-Poisson random measure
$N_0(\mathrm{d}t,\mathrm{d}z,\mathrm{d}u)$ on $\mathbb{R}_{\geq0}\times\mathbb{R}_{\geq0}$ with intensity $\mathrm{d}t m(\mathrm{d}z)$;
\item a $(\mathcal{F}_t)_{t \geq 0}$-Poisson random measure $N_1(\mathrm{d}t,\mathrm{d}z)$ on $\mathbb{R}_{\geq0}\times\mathbb{R}_{\geq0}$ with intensity $\mathrm{d}t\nu(\mathrm{d}z)$;
\item a $(\mathcal{F}_t)_{t \geq 0}$-Brownian motion $\lbrace B_t\thinspace:\thinspace t\geq0\rbrace$;
\item a $(\mathcal{F}_t)_{t \geq 0}$-Poisson random measure $M(\mathrm{d}t,\mathrm{d}z)$ on $\mathbb{R}_{\geq 0} \times \mathbb{R}$ with intensity $\mathrm{d}t \mu_E(\mathrm{d}z)$.
\end{itemize}
Suppose that these random objects are mutually independent.
Define two $(\mathcal{F}_t)_{t \geq 0}$-L\'evy processes $\lbrace \xi_t\thinspace:\thinspace t \geq 0\rbrace$ and $\lbrace Z_t\thinspace:\thinspace t \geq 0\rbrace$ by
\begin{align*}
 \xi_t &= a_E t + \sigma_E B_t + \int_0^t\int_{[-1,1]}z \widetilde{M}(\mathrm{d}s,\mathrm{d}z) + \int_{0}^t \int_{[-1,1]^c} z M(\mathrm{d}s,\mathrm{d}z),
 \\ Z_t &= b_E t + \sigma_E B_t + \int_0^t\int_{[-1,1]}(e^z - 1) \widetilde{M}(\mathrm{d}s,\mathrm{d}z) + \int_{0}^t \int_{[-1,1]^c} (e^z-1) M(\mathrm{d}s,\mathrm{d}z),
\end{align*}
where $[-1,1]^c=\mathbb{R}\backslash[-1,1]$, $\widetilde{M}(\mathrm{d}s,\mathrm{d}z):=M(\mathrm{d}s,\mathrm{d}z)-\mathrm{d}t\mu_E(\mathrm{d}z)$, and the drift coefficients $b_E$ and $a_E$ are related by
\[
b_E = a_E + \frac{\sigma_E^2}{2} + \int_{[-1,1]}( e^z - 1 - z) \mu_E(\mathrm{d}z).
\]
Note that $\lbrace Z_t\thinspace:\thinspace t \geq 0\rbrace$ has no jump less than $-1$.
According to \cite[Theorem 5.1]{MR3866603}, we have that
\begin{align}
X_t
&=
X_0 + \int_0^t \left(\beta-bX_s\right)\mathrm{d}s+\sigma\int_0^t\int_0^{\infty}\mathbbm{1}_{\lbrace u\leq X_{s}\rbrace} W\left(\mathrm{d}s,\mathrm{d}u\right) \nonumber \\
&\quad+
\int_0^t\int_0^{\infty}\int_0^{\infty}z\mathbbm{1}_{\lbrace u\leq X_{s-}\rbrace}\widetilde{N}_0\left(\mathrm{d}s,\mathrm{d}z,\mathrm{d}u\right)+
\int_0^t\int_0^{\infty}zN_1\left(\mathrm{d}s,\mathrm{d}z\right)
+ \int_0^t X_{s-}\mathrm{d}Z_s \label{eq:se cbi environment}
\end{align}
has for each $\mathcal{F}_0$-measurable random variable $X_0 \geq 0$ a pathwise unique nonnegative strong solution $\lbrace X_t\thinspace:\thinspace t\geq0\rbrace$. It is not difficult to see that the Markov generator of $\lbrace X_t\thinspace:\thinspace t\geq0\rbrace$ acting on $C_c^{\infty}(\mathbb{R}_{\geq0})$ is given by $L_0+ L_1$, where $L_0$ is defined by \eqref{eq:generator cbi} and $L_1$ by \eqref{eq:generator  CBIRE}, respectively.
In view of \cite[Theorem 5.4]{MR3866603}, the Markov process $\lbrace X_t\thinspace:\thinspace t\geq0\rbrace$ has Feller transition semigroup $\lbrace P_t\thinspace:\thinspace t \geq 0\rbrace$ and its transition probabilities $P_t(x,\mathrm{d}y)$ satisfy
\begin{align}
\int_{\mathbb{R}_{\geq 0}} \mathrm{e}^{- \lambda y}P_t(x,\mathrm{d}y)
= \mathbb{E}\left[ \exp\left( - x v_{0,t}^{\xi}(\lambda) - \int_{0}^{t} \psi\left(v_{s,t}^{\xi}(\lambda)\right) \mathrm{d}s\right) \right],\label{EQ:02}
\end{align}
where $\phi$ and $\psi$ are the corresponding branching and immigration mechanisms for the CBI process with admissible parameters $(\beta,b,\sigma,m,\nu)$
and $r \mapsto v_{r,t}^{\xi}(\lambda)$ is the pathwise unique nonnegative solution to
\[
v_{r,t}^{\xi}(\lambda) = \mathrm{e}^{\xi(t) - \xi(r)}\lambda - \int_{r}^{t}\mathrm{e}^{\xi(s) - \xi(r)}\phi\left(v_{s,t}^{\xi}(\lambda)\right) \mathrm{d}s, \quad 0 \leq r \leq t.
\]
Existence of limiting distribution was recently characterized in
\cite{MR3866603}.
In the following we present sufficient conditions for both the ergodicity in the Wasserstein and total variation distance based on an application of Theorem \ref{thm:exponential ergodicity}. We start with the simpler case of ergodicity in Wasserstein distance $W_1$.

\begin{thm}\label{thm:ergodicity in Wasserstein distance CBIRE}
Let $(\beta, b, \sigma, m, \nu)$ be admissible parameters,
$b_{E} \in \mathbb{R}$, $\sigma_{E} \geq 0$ and $\mu_E$ a L\'evy measure on $\mathbb{R}$. Let $\lbrace P_t\thinspace:\thinspace t\geq0\rbrace$ be the transition semigroup with transition probabilities defined by \eqref{EQ:02} and denote by $\lbrace P_t^{\ast}\thinspace:\thinspace t\geq0\rbrace$ the dual semigroup.
Suppose that
\begin{align}\label{EQ:03}
\int_{(1,\infty)} z \nu(\mathrm{d}z) + \int_{(1,\infty)}\mathrm{e}^{z} \mu_E(\mathrm{d}z) < \infty \quad \text{ and } \quad b > \mathbb{E}[ Z_1 ].
\end{align}
Then, for all $\varrho,\thinspace\widetilde{\varrho}\in\mathcal{P}_1(\mathbb{R}_{\geq0})$, we have
\[
W_1\left(P_t^{\ast}\varrho,P_t^{\ast}\widetilde{\varrho}\right)\leq \mathrm{e}^{-(b-\mathbb{E}[Z_1]) t} W_1\left(\varrho,\widetilde{\varrho}\right),\quad t\geq0.
\]
In particular, there exists a unique invariant distribution $\pi\in\mathcal{P}(\mathbb{R}_{\geq0})$. Moreover, $\pi$  belongs to $\mathcal{P}_1(\mathbb{R}_{\geq0})$ and, for all $\varrho\in\mathcal{P}_1(\mathbb{R}_{\geq0})$,
\[
W_1\left(P_t^{\ast}\varrho,\pi\right)\leq \mathrm{e}^{-(b-\mathbb{E}[Z_1])t}W_1\left(\varrho,\pi\right),\quad t\geq0.
\]
\end{thm}

\begin{proof}
Let us first verify that \eqref{eq:se cbi environment} is a particular case of \eqref{eq:se}. The drift coefficient is simply given by $b(x) = b_1(x) - b_2(x)$
with $b_1(x) = \beta - (b-b_E)x$ and $b_2(x) = 0$ for $x\in\mathbb{R}_{\geq0}$.
For the diffusion component set $E = \{1,2\} \times \mathbb{R}_{\geq 0}$
and $\varkappa(\mathrm{d}y,\mathrm{d}u) = \delta_1(\mathrm{d}y)\mathrm{d}u + \delta_2(\mathrm{d}y)\delta_0(\mathrm{d}u)$ on $E$.
Then $\mathcal{W}(\mathrm{d}s,\mathrm{d}y, \mathrm{d}u) := \delta_1(\mathrm{d}y)W(\mathrm{d}s,\mathrm{d}u) + \mathrm{d}B_s \delta_2(\mathrm{d}y)\delta_0( \mathrm{d}u)$
defines an $(\mathcal{F}_t)_{t \geq 0}$-Gaussian white noise on $\mathbb{R}_{\geq 0}\times E$ with intensity measure $\mathrm{d}s \varkappa(\mathrm{d}y, \mathrm{d}u)$.
Let $\sigma(x,y,u) := \sigma \mathbbm{1}_{ \{ y = 1\} } \mathbbm{1}_{ \{ u \leq x \} } + \sigma_E \mathbbm{1}_{ \{y = 2\} } \mathbbm{1}_{ \{0\}}(u)x$ for $(x,y,u)\in\mathbb{R}_{\geq0}\times E$. We see that
\[
\int_{0}^{t}\int_E \sigma\left(X_s, y, u\right)\mathcal{W}\left(\mathrm{d}s, \mathrm{d}y, \mathrm{d}u\right)
= \sigma \int_{0}^{t}\int_{\mathbb{R}_{\geq 0}}\mathbbm{1}_{ \{ u \leq X_s\} } W\left(\mathrm{d}s, \mathrm{d}u\right) + \sigma_E \int_0^t X_s \mathrm{d}B_s.
\]
Turning to the jump components, define $U_0 = \{1,2\} \times \mathbb{R} \times \mathbb{R}_{\geq 0}$
and further $\mu_0(\mathrm{d}y,\mathrm{d}z,\mathrm{d}u) = \mathbbm{1}_{\mathbb{R}_{\geq0}}(z)\delta_1(\mathrm{d}y)m(\mathrm{d}z)\mathrm{d}u + \delta_2(\mathrm{d}y)\mu_E(\mathrm{d}z)\delta_0(\mathrm{d}u)$ on $U_0$.
Then we have that $\mathcal{N}_0(\mathrm{d}s, \mathrm{d}y, \mathrm{d}z, \mathrm{d}u) := \mathbbm{1}_{\mathbb{R}_{\geq0}}(z)\delta_1(\mathrm{d}y)N_0(\mathrm{d}s, \mathrm{d}z,\mathrm{d}u)
+ \delta_2(\mathrm{d}y)M(\mathrm{d}s, \mathrm{d}z)\delta_0(\mathrm{d}u)$
defines a $(\mathcal{F}_t)_{t \geq 0}$-Poisson random measure on $\mathbb{R}_{\geq 0}\times U_0$ with intensity $\mathrm{d}s\mu_0(\mathrm{d}y,\mathrm{d}z,\mathrm{d}u)$.
Letting $g_0(x,y,z,u) = \mathbbm{1}_{ \{y = 1\}}\mathbbm{1}_{ \{ u \leq x\} }\mathbbm{1}_{\mathbbm{R}_{\geq 0}}(z) z+ \mathbbm{1}_{ \{ y = 2\} }\mathbbm{1}_{[-1,1]}(z) (\exp(z)-1)x$ for $(x,y,z,u)\in\mathbb{R}_{\geq0}\times U_0$ yields
\begin{align*}
\int_0^t \int_{U_0} g_0(X_s,y,z,u)\widetilde{\mathcal{N}}_0(\mathrm{d}s, \mathrm{d}y, \mathrm{d}z,\mathrm{d}u)
&= \int_0^t \int_0^{\infty}\int_{0}^{\infty}z \mathbbm{1}_{ \{ u \leq X_s \} } \widetilde{N}_0(\mathrm{d}s, \mathrm{d}z, \mathrm{d}u)\\
&\quad + \int_0^t \int_{-1}^{1}\left(\mathrm{e}^z - 1\right)X_s \widetilde{M}(\mathrm{d}s, \mathrm{dz})
\end{align*}
Finally, let $U_1 = \{1,2\} \times \mathbb{R}$
and define $\mu_1(\mathrm{d}y, \mathrm{d}z) = \mathbbm{1}_{\mathbb{R}_{\geq0}}(z)\delta_1(\mathrm{d}y)\nu(\mathrm{d}z) + \delta_2(\mathrm{d}y)\mu_E(\mathrm{d}z)$ on $U_1$. Then
$\mathcal{N}_1(\mathrm{d}s, \mathrm{d}y, \mathrm{d}z) = \delta_1(\mathrm{d}y)N_1(\mathrm{d}s,\mathrm{d}z) + \delta_2(\mathrm{d}y)M(\mathrm{d}s,\mathrm{d}z)$
defines a $(\mathcal{F}_t)_{t \geq 0}$-Poisson random measure on $\mathbb{R}_{\geq0}\times U_1$ with intensity
$\mathrm{d}s \mu_1(\mathrm{d}y,\mathrm{d}z)$.
Letting $g_1(x,y,z) = \mathbbm{1}_{ \{ y = 1\} } \mathbbm{1}_{\mathbb{R}_{\geq 0}}(z)z + \mathbbm{1}_{ \{ y = 2 \} } \mathbbm{1}_{[-1,1]^c}(z)(\exp(z) - 1)x$ for $(x,y,z)\in\mathbb{R}_{\geq0}\times U_1$ yields
\begin{align*}
\int_0^t \int_{U_1} g_1(X_s,y,z)\mathcal{N}_1(\mathrm{d}s, \mathrm{d}y, \mathrm{d}z)
= \int_0^t \int_0^{\infty}z N_1(\mathrm{d}s, \mathrm{d}z)
+ \int_0^t \int_{[-1,1]^c}\left(\mathrm{e}^z - 1\right)X_s M(\mathrm{d}s, \mathrm{d}z).
\end{align*}
This shows that \eqref{eq:se cbi environment} is indeed a particular case of \eqref{eq:se}. It is not difficult to see that conditions (3.a)-(3.e) are satisfied which completes the proof.
\end{proof}

\begin{thm}\label{thm:ergodicity in tv distance CBIRE}
Let $(\beta, b, \sigma, m, \nu)$ be admissible parameters,
$b_{E} \in \mathbb{R}$, $\sigma_{E} \geq 0$ and $\mu_E$ a L\'evy measure on $\mathbb{R}$.  Let $\lbrace P_t\thinspace:\thinspace t\geq0\rbrace$ be the transition semigroup with transition probabilities defined by \eqref{EQ:02} and denote by $\lbrace P_t^{\ast}\thinspace:\thinspace t\geq0\rbrace$ the dual semigroup. Suppose that \eqref{EQ:03}
and Grey's condition (5.a) is satisfied.
Let $\pi$ be the unique invariant distribution.
Then, for any $\varrho \in \mathcal{P}_{1}(\mathbb{R}_{\geq 0})$,
\begin{align}\label{EQ:07}
\| P_t^* \varrho - \pi \|_{TV} \leq 2 \mathbb{E}\left[ \overline{v}_{0,t}^{\xi} \right] W_{1}(\varrho, \pi), \quad t \geq 0,
\end{align}
where $\overline{v}_{0,t}^{\xi} := \lim_{\lambda \to \infty}v_{0,t}^{\xi}(\lambda) \in [0,\infty)$.
If, in addition, $\liminf_{t \to \infty}\xi(t) = - \infty$
almost surely, then
\begin{align}\label{EQ:06}
\lim_{t \to \infty} \left\Vert P_t^* \varrho - \pi \right\Vert_{TV} = 0.
\end{align}
\end{thm}

\begin{proof}
As a consequence of Grey's condition, \cite[Theorem 4.1]{MR3866603} applies, yielding that
$\overline{v}_{0,t}^{\xi} \in[0,\infty)$ almost surely for all $t>0$.
Let $f \in \mathcal{B}_b(\mathbb{R}_{\geq 0})$ be arbitrary.
Arguing literally as in the proof of \cite[Theorem 4.5]{MR3866603}, we observe that,
for $0 \leq x \leq y$,
\begin{align*}
\left\vert P_tf(x) - P_tf(y)\right\vert \leq 2 \Vert f \Vert_{\infty} \mathbb{E}\left[ 1 - \mathrm{e}^{-(y-x)\overline{v}_{0,t}^{\xi}}\right]
\leq 2 \Vert f \Vert_{\infty} \mathbb{E}\left [ \min\left\lbrace 1, \vert x-y\vert\overline{v}_{0,t}^{\xi}\right\rbrace \right ].
\end{align*}
Taking the supremum over all $f \in \mathcal{B}_b(\mathbb{R}_{\geq 0})$ shows that
$\Vert P_t(x,\cdot) - P_t(y,\cdot) \Vert_{TV} \leq 2 \mathbb{E}[ \min\{1,\vert x-y\vert \overline{v}_{0,t}^{\xi}\} ]$.
Let now $\varrho \in \mathcal{P}_{1}(\mathbb{R}_{\geq 0})$ and let
$H$ be any coupling of $(\varrho, \pi)$.
By convexity of the Wasserstein distance, we obtain
\begin{align*}
\left\Vert P_t^* \varrho - \pi \right\Vert_{TV}
&\leq \int_{\mathbb{R}_{\geq 0} \times \mathbb{R}_{\geq 0}} \left\Vert P_t(x,\cdot) - P_t(y,\cdot)\right\Vert_{TV} H(\mathrm{d}x,\mathrm{d}y)\\
&\leq 2 \int_{\mathbb{R}_{\geq 0}\times \mathbb{R}_{\geq 0}} \mathbb{E}\left[ \min\left\lbrace 1, \vert x-y\vert \overline{v}_{0,t}^{\xi}\right\rbrace \right] H(\mathrm{d}x,\mathrm{d}y).
\end{align*}
If $\liminf_{t \to \infty}\xi(t) = -\infty$,
then $\lim_{t \to \infty}\overline{v}_{0,t}^{\xi} = 0$ in view of \cite[Corollary 4.4]{MR3866603} and, thus, \eqref{EQ:06} follows from dominated convergence.
Finally, we conclude with the estimate \eqref{EQ:07} by estimating
\[
\int_{\mathbb{R}_{\geq 0}\times \mathbb{R}_{\geq 0}} \mathbb{E}\left[ \min\left\lbrace 1, \vert x-y\vert \overline{v}_{0,t}^{\xi}\right\rbrace \right] H(\mathrm{d}x,\mathrm{d}y)
\leq \mathbb{E}\left[ \overline{v}_{0,t}^{\xi}\right] \int_{\mathbb{R}_{\geq 0}\times \mathbb{R}_{\geq 0}} \vert x-y\vert H(\mathrm{d}x,\mathrm{d}y),
\]
where we chose $H$ as the optimal coupling of $(\varrho, \pi)$ with respect to $W_{1}$.
\end{proof}

The decay rate for $\overline{v}_{0,t}^{\xi}$ as $t\to\infty$ was studied by Palau and Pardo \cite{MR3605717} for a continuous-state branching process in Brownian random environment with stable branching. For the same class of processes but in a general L\'evy environment this problem was studied by Li and Xu \cite{MR3729532}.

\setcounter{section}{0}
\setcounter{equation}{0}
\renewcommand{\theequation}{\thesection.\arabic{equation}}
\setcounter{figure}{0}
\setcounter{table}{0}

\appendix
\section*{Appendix}
\renewcommand{\thesection}{A}
\setcounter{thm}{0}

\begin{lem}\label{lem:lyapunov estimate1}
Let $\gamma_0,\thinspace \gamma_1,\thinspace \gamma_2$ be Borel functions on $\mathbb{R}_{\geq 0}$
and $m,\thinspace \nu$ Borel measures on $\mathbb{R}_{\geq 0}$ satisfying \eqref{eq:measure conditions}. Suppose that conditions (i) -- (iii) of Theorem \ref{thm:pathwise uniqueness strong solution of nonlinear cbi}
and condition (b) of Theorem \ref{thm:exponential ergodicity in wassterstein distance for nonlinear cbi} are satisfied.
For $\lambda\in[1,2]$, define $V_{\lambda}(x) = (1 + x)^{\lambda}$, $x\geq0$. Then there exists a constant
$C > 0$ such that
\[
LV_{\lambda}(x) \leq CV_{\lambda}(x), \quad x \geq 0,
\]
where the operator $L$ is given in \eqref{EQ:NONLINEARCBI}.
\end{lem}

\begin{proof}
Defining the operators
\begin{align*}
DV_{\lambda}(x) &:= \gamma_0(x)V_{\lambda}'(x)+\frac{\gamma_1(x)}{2}V_{\lambda}''(x), \\
J_mV_{\lambda}(x) &:= \gamma_2(x)\int_{\mathbb{R}_{\geq0}}\left(V_{\lambda}(x+z)-V_{\lambda}(x)-zV_{\lambda}'(x)\right)m\left(\mathrm{d}z\right), \\
J_{\nu}V_{\lambda}(x) &:= \int_{\mathbb{R}_{\geq0}}\left(V_{\lambda}(x+z)-V_{\lambda}(x)\right)\nu\left(\mathrm{d}z\right),
\end{align*}
we see that $LV_{\lambda}=DV_{\lambda}+J_mV_{\lambda}+J_{\nu}V_{\lambda}$. Moreover, it holds that $V_{\lambda}'(x) = \lambda (1 + x)^{\lambda - 1}$ and
$V_{\lambda}''(x) = \lambda(\lambda-1)(1+x)^{\lambda - 2}$.
In the following $C > 0$ denotes some generic constant which may vary from line to line. The drift can be easily estimated by
\begin{align*}
DV_{\lambda}(x)
\leq C (1 +x)(1+x)^{\lambda - 1} + C (1+x)^{\lambda}(1+x)^{\lambda - 2}
\leq C V_{\lambda}(x).
\end{align*}
For the state-dependent jumps, by using the mean-value theorem twice, we get
\begin{align*}
V_{\lambda}(x+z) - V_{\lambda}(x) - z V_{\lambda}'(x)
= z^2\int_{0}^{1}(1-s)V_{\lambda}''(x+sz)\mathrm{d}s
\leq C z^2,
\end{align*}
and hence
$J_mV_{\lambda}(x)\leq C V_{\lambda}(x)$.
Turning to the state-independent jumps, we use the mean-value theorem to show that
\[
V_{\lambda}(x+z) - V_{\lambda}(x)
\leq C \left( \mathbbm{1}_{(0,1]}(z)z + \mathbbm{1}_{[1,\infty)}(z)z^{\lambda} \right)V_{\lambda}(x),
\]
which implies $J_{\nu}V_{\lambda}(x) \leq C V_{\lambda}(x)$.
Collecting the estimates for $DV_{\lambda}$, $J_mV_{\lambda}$, and $J_{\nu}V_{\lambda}$ proves the asserted.
\end{proof}

\begin{lem}\label{lem:lyapunov estimate}
Let $(\beta, b, \sigma, m, \nu)$ be admissible parameters
and suppose that \eqref{eq:inequality w_log pi} holds.
Let $V(x):=\log(1+x)$, $x\geq0$. Then there exists a constant $C>0$ such that
\[
LV(x)\leq C,\quad x\geq0,
\]
where the operator $L$ is given in \eqref{eq:generator cbi}.
\end{lem}

\begin{proof}
Let us introduce the operators
\begin{align*}
DV(x) &:= \left( \beta-bx\right)V'(x)+\sigma xV''(x);\\
J_mV(x) &:= x\int_0^{\infty}\left(V(x+z)-V(x)-zV'(x)\right)m\left(\mathrm{d}z\right);\\
J_{\nu}V(x) &:= \int_0^{\infty}\left(V(x+z)-V(x)\right)\nu\left(\mathrm{d}z\right).
\end{align*}
So $LV=DV+J_mV+J_{\nu}V$. We now estimate $DV$, $J_mV$, and $J_{\nu}V$ separately. Concerning $DV$, it is easy to see that
$DV(x)\leq \frac{\beta}{1+x}-\frac{bx}{1+x}\leq \beta + |b|$.
Turning to $J_{\nu}V$, we first note that
\begin{align}\label{EQ:00}
V(x+z)-V(x)=\log\left(1+\frac{z}{1+x}\right)\leq\frac{z}{1+x}\leq z
\end{align}
on the one hand and
\[
V(x+z)-V(x)=\log\left(1+\frac{z}{1+x}\right)\leq \log(1+z)
\]
on the other hand.
Having established the latter inequalities, we obtain
\begin{equation*}
J_{\nu}V(x)\leq \int_0^{\infty}\left( z\mathbbm{1}_{\lbrace 0< z\leq1\rbrace}
+\log(1+z)\mathbbm{1}_{\lbrace z>1\rbrace}\right)\nu\left(\mathrm{d}z\right)<\infty.
\end{equation*}
For $J_mV$, we decompose it further as $J_mV=J_{m,\ast}V+J_m^{\ast}V$, where
\begin{align*}
J_{m,\ast}V(x)&:=x\int_{\lbrace 0<z\leq 1\rbrace}\left(V(x+z)-V(x)-zV'(x)\right)m\left(\mathrm{d}z\right),\\
J_m^{\ast}V(x)&:= x\int_{\lbrace z>1\rbrace}\left(V(x+z)-V(x)-zV'(x)\right)m\left(\mathrm{d}z\right).
\end{align*}
For $J_{m}^*$ we use \eqref{EQ:00} to obtain
\[
J_m^{\ast}V(x)
\leq \frac{x}{1+x}\int_{\lbrace z>1\rbrace}z m\left(\mathrm{d}z\right)
\leq \int_{\lbrace z>1\rbrace}zm\left(\mathrm{d}z\right)<\infty.
\]
For $J_{m,*}$ we use the mean value theorem and
\[
 V'(x+rz) - V'(x) = \frac{1}{x+rz} - \frac{1}{x} = - \frac{rz}{x(x+rz)}
 \leq 0,
\]
where $r \in [0,1]$, to obtain
\begin{align*}
 J_{m,\ast}V(x)
&= x\int_{\lbrace 0<z\leq 1\rbrace}\left(\int_0^1\left(V'(x+rz)-V'(x)\right)z\mathrm{d}r\right)m\left(\mathrm{d}z\right) \leq 0.
\end{align*}
Combining the estimates for $DV, \thinspace J_{m,*}V, \thinspace J_m^*V$, and $J_{\nu}V$
yields the asserted estimate.
\end{proof}

\def\cprime{$'$} \def\cprime{$'$}
\providecommand{\bysame}{\leavevmode\hbox to3em{\hrulefill}\thinspace}
\providecommand{\MR}{\relax\ifhmode\unskip\space\fi MR }
\providecommand{\MRhref}[2]{%
  \href{http://www.ams.org/mathscinet-getitem?mr=#1}{#2}
}
\providecommand{\href}[2]{#2}

\end{document}